%% file: thin_film_paper.tex
\documentclass[]{amsart}
\usepackage[utf8]{inputenc}
\usepackage[UKenglish]{babel}
\usepackage{amsfonts}
\usepackage{amssymb}
\usepackage{amsmath}
\usepackage{setspace}
\usepackage{subfigure}
\usepackage{hyperref}
\usepackage{graphicx}
\usepackage{multirow}

\usepackage{url}
\usepackage{cite}

\include{dgmacros}
\include{gsmacros}

\begin{document}

\title[Steady State Solutions of Thin Films]{Analytical and Numerical
  Results on the Positivity of Steady State Solutions of a Thin Film
  Equation}

\author[Ginsberg]{Daniel Ginsberg} \email{dan.ginsberg@utoronto.ca}
\address{Department of Mathematics\\University of Toronto\\Toronto,
  ON, M5S 2E4\\Canada}

\author[Simpson]{Gideon Simpson} \email{gsimpson@math.umn.edu}
\address{School of Mathematics\\University of Minnesota\\Minneapolis,
  MN, 55455\\USA}

\date{\today}


\maketitle

\begin{abstract}

  We consider an equation for a thin-film of fluid on a rotating
  cylinder and present several new analytical and numerical results on
  steady state solutions.  First, we provide an elementary proof that
  both weak and classical steady states must be strictly positive so
  long as the speed of rotation is nonzero.  Next, we formulate an
  iterative spectral algorithm for computing these steady states.
  Finally, we explore a non-existence inequality for steady state
  solutions from the recent work of Chugunova, Pugh, \& Taranets.

\end{abstract}

\section{Introduction}

Thin liquid films appear in a wide range of natural and industrial
applications, such as coating processes, and are generically
characterized by a small ratio between the thickness of the fluid and
the characteristic length scale in the transverse direction.  The
physics of such processes are reviewed in Oron, Davis \&
Bankoff,\cite{Oron:1997p8421} and more recently in Craster \& Matar,
\cite{Craster:2009p8507}.  If one includes surface tension in the
model, the fluid is governed by a degenerate fourth order parabolic
equation.  Such equations, studied by Bernis \& Friedman,
\cite{Bernis1990}, Bertozzi \& Pugh, \cite{Bertozzi:1994p10686}, and
Beretta, Bertsch \& dal Passo \cite{Beretta:1995p13153} still hold
many challenges. See, for example,
\cite{Shishkov:2004p13213,Taranets:2006p13207, Chugunova:2010p6065}
for progress on models which include convection, and Becker \&
Gr\"un,\cite{Becker:2005p8459}, for a thorough review of analytical
progress.

There has been recent interest in equations governing the evolution of
a thin, viscous film on the outer (or inner) surface of a cylinder of
radius $R$ rotating with constant angular velocity $\omega$, as in
Figure \ref{fig:cyl}.  Subject to non-dimensionalization, this can be
modeled by
\begin{equation}
  u_t + (u^n(\ut + u_{\theta} - \sin\theta) + \omega u)_{\theta} = 0,
  \label{eq:tfe}
\end{equation}
Of particular interest is the case $n = 3$, in which one assumes there
is no slip at the fluid-cylinder interface.  The solution, $u$,
measures the (dimensionless) thickness of the fluid, parameterized by
the angular variable, $\theta\in \Omega \equiv [-\pi, \pi)$.  To
derive \eqref{eq:tfe}, one begins with the Navier-Stokes equations,
together with a kinematic boundary condition for the free surface
surface tension.  Subject to appropriate physical scalings and
geometric symmetries, we are left with a dimensionless equation with
just the $\omega$ parameter.  We refer the reader to, for instance,
\cite{Pukhnachev1977} and \cite{Pougatch:2011p14442} for a full
derivation.

\begin{figure}[h!]
  \begin{center}
    \leavevmode
    \includegraphics[scale=0.3]{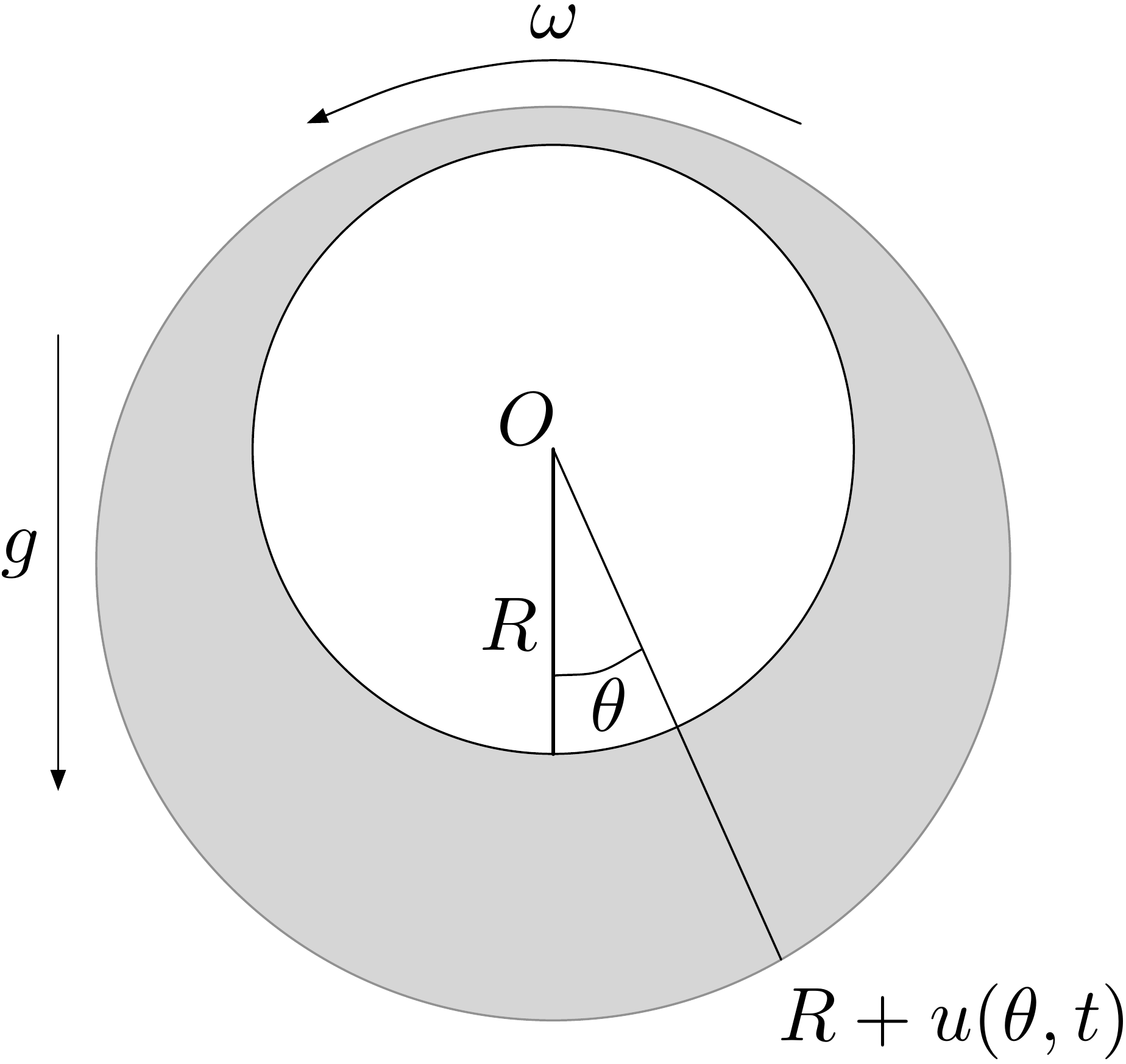}
  \end{center}
  \caption{ The grey region represents the thin film of liquid on the
  surface of the cylinder, which is rotating with constant angular velocity
  $\omega$. We assume symmetry in the $z$ coordinate, so that 
  $u$ is a function of $t$ and $\theta$ only.
}

\label{fig:cyl}
\end{figure}

Though Buchard, Chugunova, \& Stephens,\cite{Burchard:2010p8420}, gave
a detailed examination of \eqref{eq:tfe} when $\omega=0$, our
understanding of the case $\omega \neq 0$ remains incomplete. Some
partial results on this case have made assumptions that solutions of
\eqref{eq:tfe} are strictly positive, such as in
\cite{Pukhnachov:2005p5033}, where Pukhnachov proved, under a
smallness assumption, positive solutions of \eqref{eq:tfe} exist and
are unique.

In this work, we further study steady state solutions -- {\it
  non-negative}, {\it time independent} solutions $u(\theta)$ of
\eqref{eq:tfe}, satisfying
\begin{equation}
  (u^n(\ut + u_{\theta} - \sin\theta) + \omega u)_{\theta} = 0
  \label{eq:steady}
\end{equation}
This can be integrated once, to
\begin{equation}
  u^n(\ut +u_{\theta} -\sin\theta) + \omega u = q,
  \label{eq:steadyq}
\end{equation}
where $q$, the constant of integration, is the dimensionless {\it
  flux} of the fluid.  In particular, we give an elementary proof that
when $\omega \neq 0$, both weak and classical solutions must be
strictly positive; see Section \ref{s:positivity}.

For a more extended study of the steady states, it is helpful to
explore the problem numerically.  This can provide information on how,
and if, $u$ and its derivatives develop singularities as the rotation
parameter, $\omega$, goes to zero. For such a task, one needs a robust
algorithm for computing solutions of \eqref{eq:steadyq}.  In Section
\ref{s:algorithm}, we formulate and benchmark an iterative spectral
algorithm that is easy to implement and rapidly converges.

Our last result, appearing in Section \ref{s:non_exist}, is motivated
by Chugunova, Pugh, \& Taranets, \cite{Chugunova:2010p6065}, who
proved that for $n=3$ there can be no positive steady state solutions
when the rotation and flux parameters satisfy the inequality
\begin{equation}
  \label{e:nonexist}
  q > \paren{\tfrac{2}{3}}^{3/2} \omega^{3/2}.
\end{equation}
This was refinement of an earlier bound due to Pukhnachov,
\cite{pukhnachov2004asymptotic}.  We explore the $(\omega, q)$ phase
space and find evidence that this bound may be sharp.  Furthermore, it
appears that this curve is related to a transition between two
different regimes of steady state solutions.

\subsection{Notation} 
In what follows, we shall use the following notational conventions.
$C^{k}_{\per}(\Omega)$ is the set of continuous $2\pi$-periodic
functions on $\Omega = [-\pi, \pi)$ with $k$ continuous derivatives.
$H^k_{\per}(\Omega)$ is the set of square integrable $2\pi$-periodic
functions on $\Omega = [-\pi, \pi)$ with $k$ square integrable weak
derivatives.  For a function $\phi$ in one of these spaces, we shall
denote its mean by $\bar{\phi}$.  For brevity, we will also sometimes
write
\[
a \lesssim b
\]
to indicate that there exists a constant $C>0$ such that
\[
a \leq C b.
\]

\section{Positivity of Steady States}
\label{s:positivity}
	
Here, we prove that both classical and weak steady state solutions of
\eqref{eq:tfe} are strictly positive provided $\omega \neq 0$.
Buchard, Chugunova, \& Stephens, \cite{Burchard:2010p8616},
demonstrated that for $\omega = 0$, such solutions need not be
positive.

\subsection{Classical Steady State Solutions}

As noted, the case $n = 3$ corresponds to a no-slip condition at the
surface of the cylinder.  With this in mind, it is physically obvious
that, as long as $\omega \neq 0$, any steady state solution of
\eqref{eq:tfe} with $n = 3$ will coat the entire cylinder -- that is,
there will be no ``dry patches'' on the surface of the cylinder. The
following elementary argument shows that this physical intuition is
correct, and is in fact valid for a broad range of exponents $n$.

\begin{prop}[Positivity of Classical Steady State Solutions]
  \label{p:classical}
  Let $u\in C^4_{\per}(\Omega)$ be a classical steady-state solution
  of \eqref{eq:tfe} with $n>1$.  If $\supp(u)$ $\neq \Omega$, then
  $\omega = 0$.

\end{prop}
\begin{proof}
  Assume that $u$ vanishes at $\theta_0$.  Evaluating
  \eqref{eq:steadyq} at $\theta_0$, we have $q=0$. Consequently,
  \begin{equation}
    u^n(\ut+u_{\theta}-\sin\theta) + \omega u = 0.
    \label{eq:0flux}
  \end{equation}
  On the support of $u$,
  \begin{equation}
    u^{n-1}(\ut  +u_{\theta}  -\sin\theta)+\omega = 0.
    \label{eq:divide}
  \end{equation}
  Since $\ut$ is bounded as we send $\theta$ towards the root, $\omega
  =0$.
\end{proof}
Though a classical solution of \eqref{eq:tfe} needs to be $C^4$, the
above proof succeeds without change if $u$ is merely $C^2$ with
bounded third derivative.

\subsection{Weak Steady State Solutions}
Positivity can also be proven under weaker assumptions on $u$. We call
a non-negative function $u\in H^2_\per(\Omega)$ a {\it weak solution}
of\eqref{eq:steady}, if for all $\phi \in \fns$,
\begin{equation}
  \mint{ (u^n\phi_{\theta})_{\theta}(\term) - \omega u\phi_{\theta}} = 0.
  \label{eq:1w}
\end{equation}

We first establish some pointwise properties of a solution $u$, in the
neighborhood of a hypothetical touchdown point, $\theta_0$.
\begin{lem}
  \label{lem:pointwise}
  Let $u\in H^2_\per(\Omega)$ be non-negative.  If there exists
  $\theta_0$ such that $u(\theta_0)=0$, then
  \begin{subequations}
    \label{e:pointwise_est}
    \begin{align}
      u'(\theta_0) = 0\\
      u(\theta) = \abs{u(\theta)- u(\theta_0)} &\lesssim
      \abs{\theta-\theta_0}^{3/2},\\
      \abs{u'(\theta)} = \abs{u'(\theta)- u'(\theta_0)} &\lesssim
      \abs{\theta-\theta_0}^{1/2}.
    \end{align}
  \end{subequations}
\end{lem}

\begin{proof}
  By virtue of a Sobolev embedding theorem in dimension one,
  $H^2\hookrightarrow C^{1,1/2}$,
  \cite{adams1975sobolev,Evans:1998fk}.  Therefore, both $u$ and $u'$
  are H\"older continuous with exponent $1/2$;
  \begin{subequations}
    \label{e:pointwise_est1}
    \begin{align}
      \abs{u(\theta)- u(\theta_0)} &\lesssim
      \abs{\theta - \theta_0}^{1/2}\\
      \abs{u'(\theta)- u'(\theta_0)} &\lesssim \abs{\theta -
        \theta_0}^{1/2}.
    \end{align}
  \end{subequations}
  Because $u$ is $C^1$ {\it and} non-negative, we must have that
  $u'(\theta_0) = 0$ too, otherwise there would be a zero crossing.

  Refining \eqref{e:pointwise_est1}, we can apply the mean value
  theorem and the H\"older continuity of $u'$ to get
  \[
  \abs{u(\theta)- u(\theta_0)} =
  \abs{u'(\theta_\star)}\abs{\theta-\theta_0} =\abs{u'(\theta_\star) -
    u'(\theta_0)}\abs{\theta-\theta_0}\lesssim
  \abs{\theta-\theta_0}^{3/2},
  \]
  where $\theta_\star$ lies between $\theta$ and $\theta$.
\end{proof}

To prove the weak form of Proposition \ref{p:classical}, we begin by
deriving a weak analog of \eqref{eq:steadyq}:
\begin{lem}
  \label{l:weakflux}
  If $u$ is a weak solution in the sense of \eqref{eq:1w}, then there
  is a constant $q$, depending only on $u$, such that
  \begin{equation}
    \mint{(u^n\phi)_{\theta}(\term)-\omega u \phi} = q\mint{\phi} 
    \label{eq:2w}
  \end{equation}
  for all $\phi\in\fns$.
\end{lem}

\begin{proof}
  First, observe that for any $\phi \in\fns$, $\phi - \bar{\phi}$ has
  mean zero and therefore has an antiderivative in $\fns$. Given a
  test function $\phi$, let $\varphi$ solve $\varphi_{\theta} =
  \phi-\bar{\phi}$.  Plugging in $\varphi$ in \eqref{eq:1w} in place
  of $\phi$, we have
  \begin{equation}
    \mint{(u^n\phi)_{\theta}(\term) -\omega u \phi} =   \bp \mint{(u^n)_{\theta}(\term) - \omega u}.
  \end{equation}
  We may now take
  \[
  q = 2\pi \mint{(u^n)_{\theta}(\term) - \omega u}.
  \]

\end{proof}

We next derive a weak analog of equation \eqref{eq:0flux} in the event
that $u$ has a root.
\begin{lem}
  \label{l:weakzeroflux}
  For $n>1$, let $u$ be a weak solution satisfying \eqref{eq:2w} for
  all $\phi\in\fns$. If $u(\theta_0) = 0$ for some
  $\theta_0\in\Omega$, then $q = 0$.
\end{lem}

\begin{proof}

  Given $\varepsilon > 0$, let $\phi \in \fns$ be a compactly
  supported test function satisfying:
  \begin{equation}
    \label{e:bump_func}
    0 \leq \phi \leq 1, \quad \supp\phi \subseteq (-1, +1), \quad \mint{\phi} = 1.
  \end{equation}
  Let
  \begin{equation}
    \label{e:bump_rescaled}
    \phi^\vareps(\theta) \equiv \vareps^{-1} \phi((\theta-\theta_0)/\vareps)
  \end{equation}
  Then $\supp\phi^\vareps \subseteq I_\varepsilon \equiv
  [\theta_0-\varepsilon, \theta_0+\varepsilon]$.  Substituting
  $\phi^\vareps$ into \eqref{eq:2w}, we compute
  \begin{equation}
    \label{e:q_bound}
    \begin{split}
      \abs{q} &\leq \abs{\omega}\int_{I_\varepsilon} u \phi^\vareps d
      \theta + \int_{I_\varepsilon} \abs{n u^{n-1} u_\theta
        \phi^\vareps + u^n
        \phi_\theta}\abs{u_{\theta\theta} + u - \sin\theta}d \theta\\
      &\leq \abs{\omega}\norm{u}_{L^\infty(I_\varepsilon)}+ \norm{n
        u^{n-1} u_\theta \phi^\vareps + u^n
        \phi^\vareps_\theta}_{L^2(I_\varepsilon)}
      \norm{u_{\theta\theta}
        + u - \sin\theta}_{L^2(I_\varepsilon)} \\
      &\leq \abs{\omega}\norm{u}_{L^\infty(I_\varepsilon)} +
      C\sqrt{2\vareps}\paren{\norm{n u^{n-1} u_\theta \phi^\vareps
        }_{L^\infty(I_\varepsilon)} + \norm{u^n
          \phi^\vareps_\theta}_{L^\infty(I_\varepsilon)}}
    \end{split}
  \end{equation}
  The constant $C$ depends on the $L^2$ norms of $u$ and its second
  derivative, which are both finite by assumption.
  
  Using estimates \eqref{e:pointwise_est}, we have that within
  $I_\varepsilon$,
  \begin{subequations}
    \label{e:holder_est}
    \begin{align}
      u^n\phi_{\theta}^\vareps &\lesssim \varepsilon^{\frac{3n}{2}-2}\\
      nu^{n-1}u_{\theta}\phi^\vareps &\lesssim
      \varepsilon^{\frac{3n}{2}-2}
    \end{align}
  \end{subequations}
  Substituting this pointwise analysis into \eqref{e:q_bound},
  \begin{equation}
    \abs{q} \lesssim \varepsilon^{\frac{3}{2}} + \varepsilon^{(3n-3)/2}.
  \end{equation}
  Since this holds for all $\varepsilon>0$, we conclude that $q=0$ for
  $n>1$.

\end{proof}

Using Lemmas \ref{l:weakflux} and \ref{l:weakzeroflux}, we now know
that a weak solution with a touchdown satisfies
\begin{equation}
  \mint{(\phi u^n)_{\theta}(\term)-\omega u \phi} = 0.
  \label{eq:3w}
\end{equation}
Contrast this with its classical counterpart, equation
\eqref{eq:0flux}.

We now proceed to our main result in the weak case, under the
additional restriction that $n \geq 2$:
\begin{prop}[Positivity of Weak Solutions]
  Let $u$ be a weak solution satisfying \eqref{eq:1w} with $n \geq 2$
  for all $\varphi\in\fns$. If $S = \{\theta : u(\theta) > 0\} \neq
  \Omega$, then $\omega = 0$.
\end{prop}
\begin{proof} The proof is by contradiction.  Assuming that $S \neq
  \Omega$, let $\theta_0$ be a point on the boundary of $S$.  Given
  $\varepsilon > 0$ sufficiently small, take $y\in \Omega$ such that
  the interval $I_{\varepsilon/2} \equiv
  [y-\tfrac{\varepsilon}{2},y+\tfrac{\varepsilon}{2}]$ is contained in
  $S$ and $|\theta_0-y| =\varepsilon$.  Now, let $\psi$ be a compactly
  supported test function as in \eqref{e:bump_func}, and let
  \[
  \psi^\vareps(\theta) = 2\vareps^{-1}\psi(2(\theta - y)/\vareps).
  \]
  This function then has support in $I_{\varepsilon/2}$ as in Figure
  \ref{f:bump}.  Finally, define $\phi^\vareps = \psi^\vareps/u^n$,
  which is as smooth as $u$, by the conditions on $\varepsilon$ and
  $y$.

  Plugging $\phi^{\varepsilon}$ into \eqref{eq:3w}, we see:
  \begin{equation}
    \begin{split}
      \abs{\omega} \int_{I_{\varepsilon/2}} {u}^{1-n}{\psi^\vareps}
      d\theta &\leq \int_{I_{\varepsilon/2}}\abs{\psi_\theta^\vareps}
      \abs{u_{\theta\theta} + u -
        \cos\theta} d\theta\\
      &\leq\norm{\psi_\theta^\vareps}_{L^2(I_{\varepsilon/2})}
      \norm{u_{\theta\theta} + u -
        \cos\theta}_{L^2(I_{\varepsilon/2})}.
    \end{split}
  \end{equation}
  Since ${u}\lesssim \varepsilon^{\frac{3}{2}}$ on the interval and
  $n>1$, we can bound the integral on the left from below:
  \begin{equation}
    \int_{I_{\varepsilon/2}} {u}^{1-n}{\psi^\vareps} d\theta  \gtrsim  \varepsilon^{\frac{3}{2}(1-n)}.
  \end{equation}
  By the construction of $\psi^\vareps$,
  $\norm{\psi_\theta^\vareps}_{L^2(I_{\varepsilon/2})} \lesssim
  \varepsilon^{-3/2}$.  Therefore,
  \begin{equation}
    \abs{\omega} \lesssim  \varepsilon^{\frac{3}{2}(n-2)}\norm{u_{\theta\theta} + u -
      \cos\theta}_{L^2(I_{\varepsilon/2})}
  \end{equation}
  For $n>2$, we immediately observe that since $\varepsilon>0$ is
  arbitrary, $\omega = 0$.  For $n=2$, we know that since
  $(u_{\theta\theta})^2$ and $u^2$ are measurable,
  \[
  \norm{u_{\theta\theta} + u - \sin\theta}_{L^2(I_{\varepsilon/2})}
  \to 0
  \]
  as $\varepsilon \to 0$, which implies $\omega = 0$ in this case.

  \begin{figure}[h!]
    \begin{center}
      \leavevmode
      \includegraphics[width=3in]{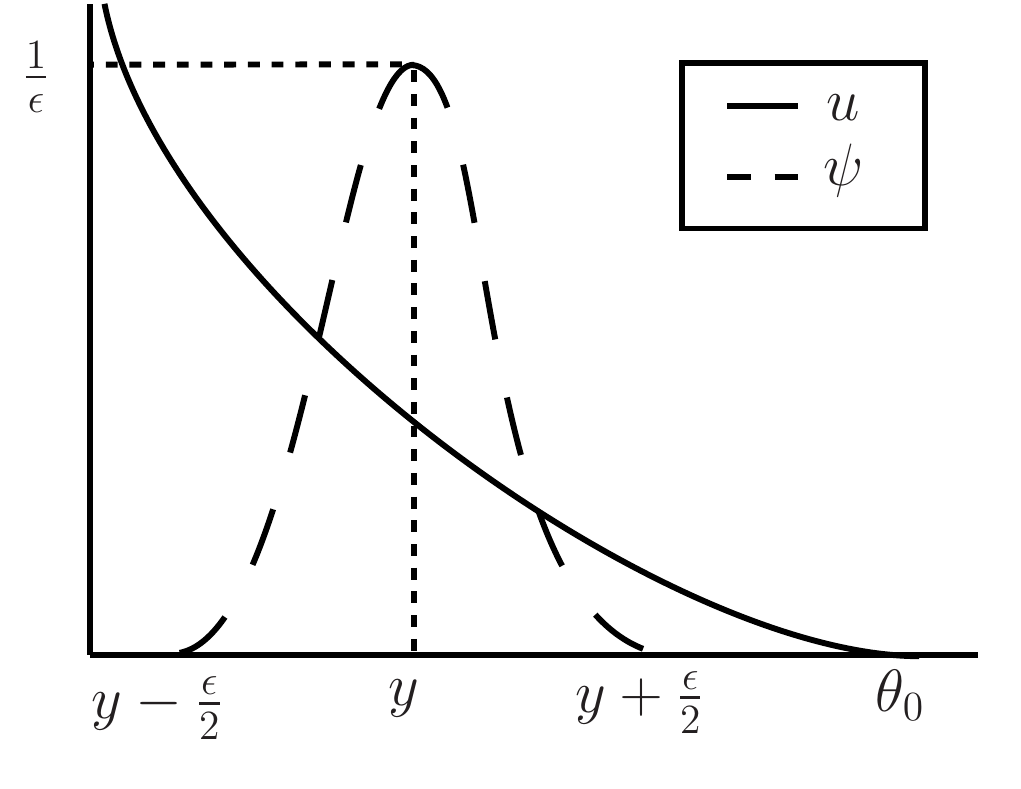}
    \end{center}
    \caption{In the proof of proposition 2, we choose a smooth,
      compactly supported function $\psi$ with support contained in
      the support of $u$, but close to a zero of $u$.}
    \label{f:bump}
  \end{figure}

\end{proof}

\section{A Computational Method for Steady State Solutions }
\label{s:algorithm}

While we have given a simple proof that when $\omega \neq 0$,steady
state solutions must be positive, we have not established their
existence.  For that, we rely on \cite{Pukhnachov:2005p5033}, where it
was shown that in the case $n=3$ ``small solutions'' existed and were
unique.  Let us motivate this with some heuristic asymptotics.  First,
assume that $0 < q \ll \omega$, {\it i.e.}
\begin{equation}
  \label{e:puk_regime}
  q = \eps \omega, \quad 0<\eps \ll 1.
\end{equation}
Subject to this assumption, we may then perform a series expansion in
$\eps$,
\begin{equation}
  u = u^{(0)} + \eps u^{(1)} + \eps^2 u^{(2)} + \eps^3 u^{(3)} +
  \ldots
\end{equation}
Substituting into
\begin{equation}
  u^3(u_{\theta\theta\theta} + u_\theta - \sin\theta) + \omega u = \eps \omega,
\end{equation}
we then match orders of $\eps$.  At order $\eps^{0}$,
\begin{equation}
  (u^{(0)})^3 ( \partial_\theta^n u^{(0)} + \partial_\theta
  u^{(0)} - \sin\theta) + \omega  u^{(0)} = 0.
\end{equation}
As we have no way to solve this, we shall assume the order
$\bigo(\eps^{0})$ term is trivial, $u^{(0)}=0$; consequently, $u$ is
$\bigo(\eps^{1})$. Proceeding with this assumption, the next few
orders are:
\begin{subequations}
  \begin{align}
    \bigo\paren{\eps^{1}}:&\quad \omega u^{(1)} = \omega,\\
    \bigo\paren{\eps^{2}}:& \quad \omega u^{(2)} = 0,\\
    \bigo\paren{\eps^{3}}:& \quad -(u^{(1)})^3 \sin\theta + \omega
    u^{(3)}=0.
  \end{align}
\end{subequations}
The leading order approximation of the steady state is thus
\begin{equation}
  \label{e:approx_steady}
  u \approx \frac{q}{\omega} + \frac{1}{\omega}\paren{\frac{q}{\omega}}^3
  \sin\theta. 
\end{equation}
This corresponds to the ``strongly supercritical regime'' discussed by
Benilov {\it et al.},\cite{Ashmore:2003p8741, Benilov:2008p4148}.  A
similar analysis applies to the case $n=2$, for which an approximate
steady state solution is
\begin{equation}
  \label{e:approx_steadyn2}
  u \approx \frac{q}{\omega} + \frac{1}{\omega}\paren{\frac{q}{\omega}}^2
  \sin\theta. 
\end{equation}
This regime was studied in \cite{Pukhnachov:2005p5033}, where the
expansion was justified by a contraction mapping argument.  This
motivates the numerical implementation of this approach as a tool for
computing such steady states.

\subsection{Description of the Algorithm}

In what follows, we focus on the case $n = 3$, although it can be
easily adapted to other cases.  We seek to rewrite \eqref{eq:steadyq}
in the form
\[
u = L^{-1} f(u) + g(\theta)
\]
where $L$ is a linear operator, $f$ contains the nonlinear terms and
$g$ is a a driving term.  This will motivate the iteration scheme
\[
u^{(j+1)} = L^{-1} f(u^{(j)}) + g(\theta)
\]
To begin, we divide \eqref{eq:steadyq} by $u^3$, to obtain
\[
(\partial_{\theta}^3 + \partial_\theta) u = \sin\theta + q u^{-3} -
\omega u^{-2}.
\]
As $\partial_{\theta}^3 + \partial_\theta$ has a non trivial kernel,
the lefthand side cannot be inverted without careful
projection. Pukhnachov resolved this problem by adding and subtracting
appropriate terms to the equation.  First, we introduce the variable
$v$,
\[
v \equiv u - \frac{q}{\omega}.
\]
Then
\begin{equation}
  \label{e:v_steadystate}
  \begin{split}
    L v \equiv (\partial_{\theta}^3
    + \partial_\theta+\frac{\omega^4}{q^3}) v &= \sin\theta -
    \frac{\omega v}{(v
      + \frac{q}{\omega})^3} + \frac{\omega^4}{q^3} v\\
    & = \sin\theta + \frac{\omega^5}{q^3}v^2 \frac{3 q^2+ 3 q\omega v+
      \omega^2 v^2}{(q + v\omega)^3}\equiv \sin\theta+ F(v)
  \end{split}
\end{equation}
The right hand side of \eqref{e:v_steadystate} contains a driving
term, $\sin\theta$, and an essentially nonlinear term of order
$\bigo(v^2)$.  Rearranging \eqref{e:v_steadystate},
\begin{equation}
  v = \frac{q^3}{\omega^4}\sin\theta + L^{-1}{F}( v)
  \label{eq:v_ft}
\end{equation}
We now have our iteration scheme
\begin{equation}
  \label{e:iter}
  v^{(j+1)} = \frac{q^3}{\omega^4}\sin\theta + L^{-1}{F}( v^{(j)}).
\end{equation}
In Fourier space, this is
\begin{equation}
  \ft{v}^{(j+1)}_k =\frac{q^3}{\omega^4} \frac{\delta_{k,1} -
    \delta_{k,-1}}{2i}+
  \frac{\ft{F}_k^{(j)}}{-ik^3+ik+\frac{\omega^4}{q^3}}, \quad F^{(j)} =
  F(v^{(j)})
  \label{eq:pvviter}
\end{equation}
Plots of solutions computed using this algorithm appear in Figures
\ref{f:plots} and \ref{f:converge}.

\begin{figure}[h]
  \centering \subfigure[Solutions with fixed $\omega =
  0.1$]{\includegraphics[width=2.45in]{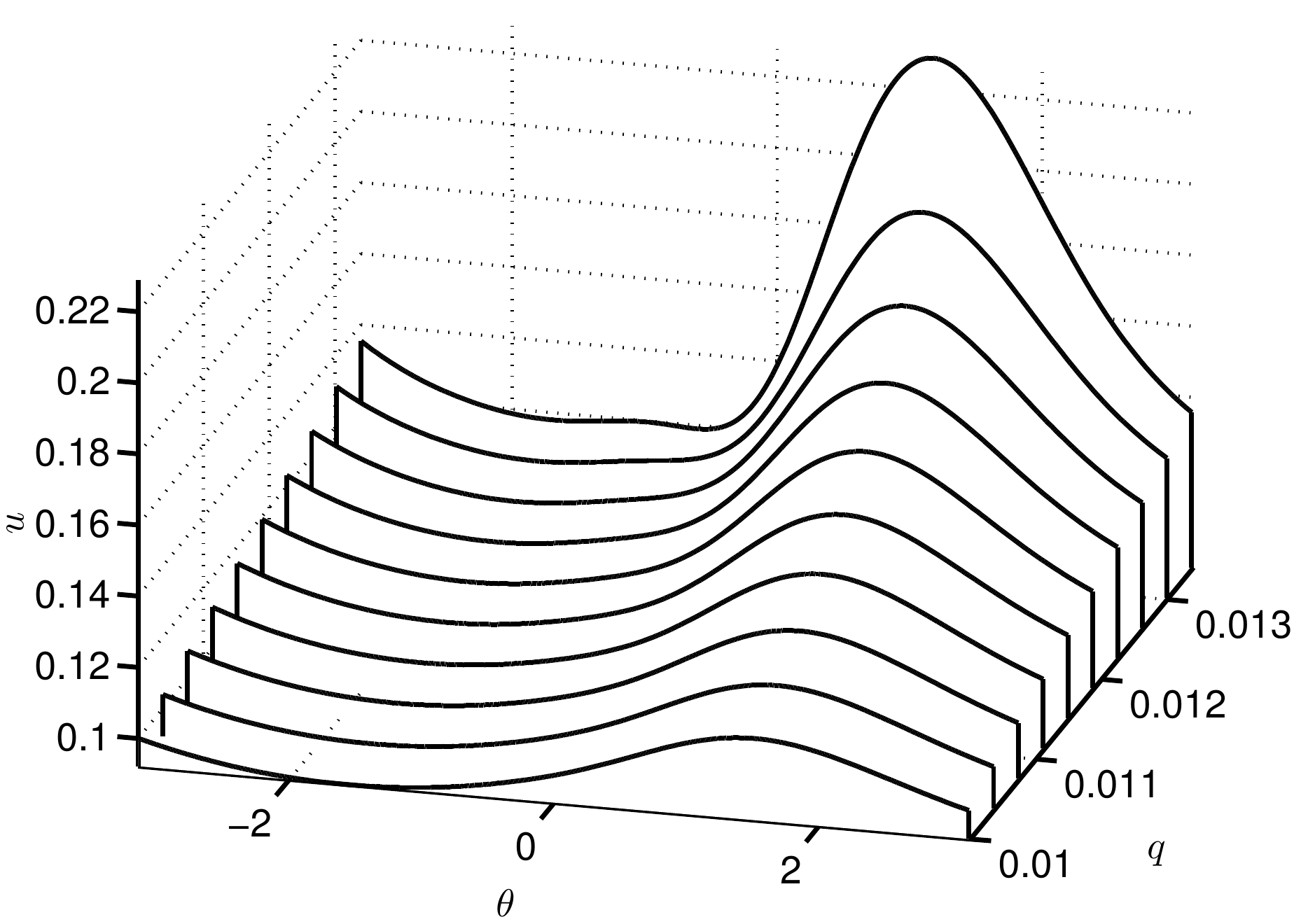}}
  \subfigure[Solutions with fixed $q =
  0.01$]{\includegraphics[width=2.45in]{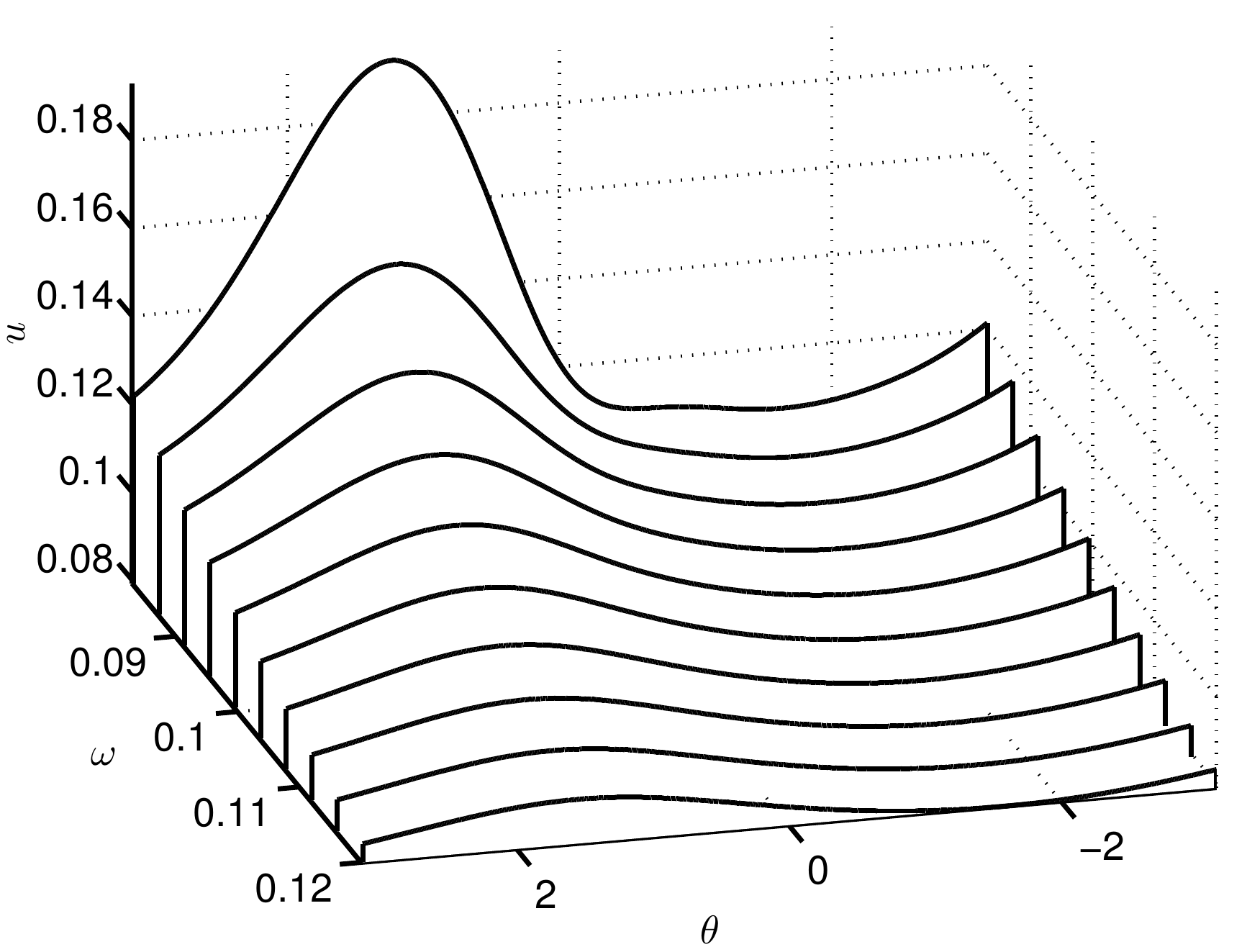}}
  \caption{Steady state solutions with various $\omega$ and $q$
    values.  Computed by the iterated spectral algorithm at 512 grid
    points to a tolerance of $10^{-15}$.}
  \label{f:plots}
\end{figure}

\begin{figure}[h]
  \begin{center}
    \includegraphics[width=2.47in]{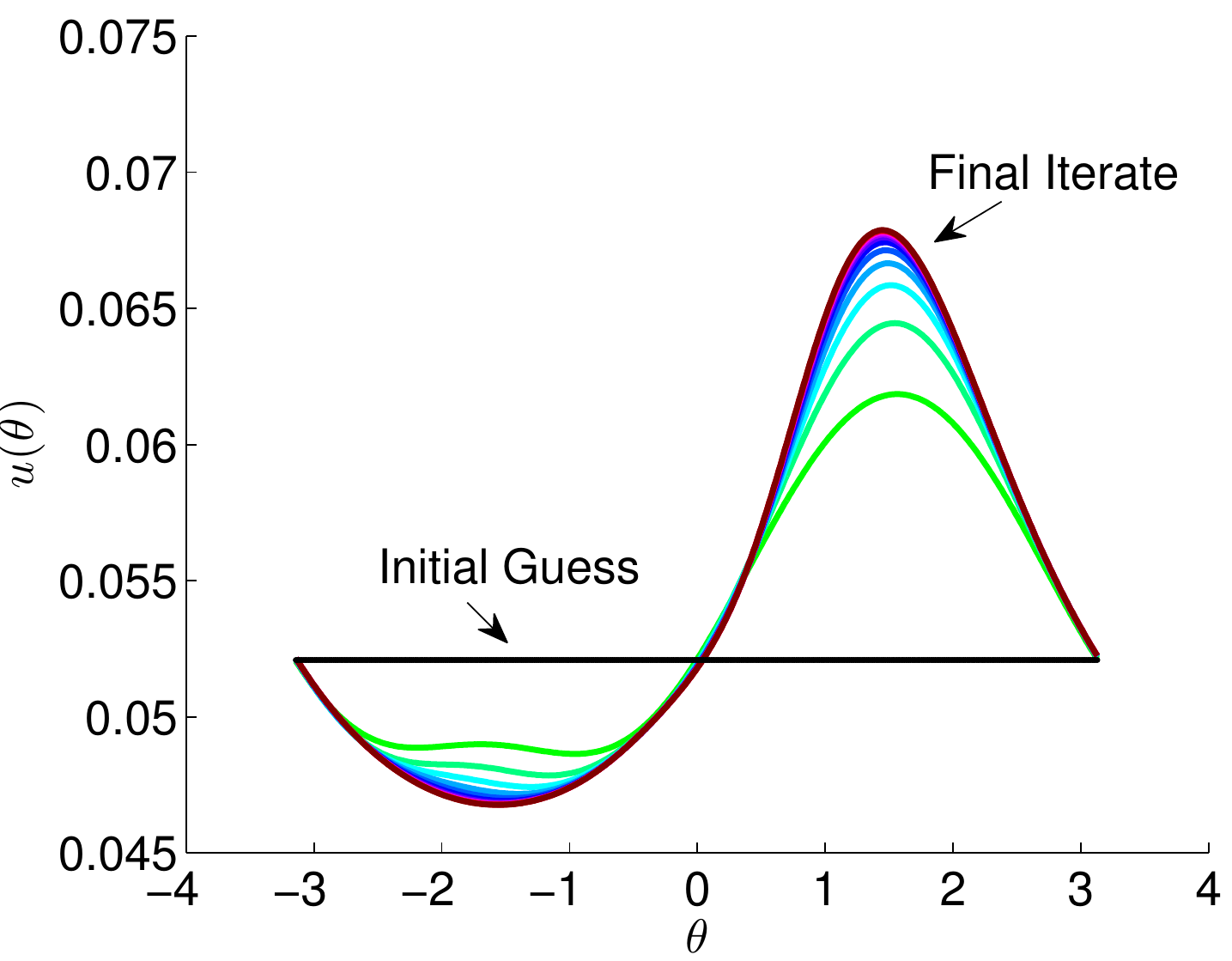}
  \end{center}
  \caption{Convergence to a steady state, with $v^{(0)} = 0$, $\omega
    = 0.0192$ and $q = 0.01$.}
  \label{f:converge}
\end{figure}

\subsection{Performance of the Spectral Iterative Method}

When our algorithm successfully converges, it does so quite rapidly,
as shown in Figure \ref{f:iterate_errors}, where we plot the error
between iterates $v^{(j)}$ and the final iterate. We see that the
error is proportional to $e^{-\alpha N}$, although it is clear that
the rate, $\alpha$, varies with $q$ and $\omega$.  Indeed, as shown in
Figure \ref{f:diverge}, there are choices of these parameters for
which our algoirthm diverges; this is consistent with the threshold
\eqref{e:nonexist}.  While this will be further discussed in Section
\ref{s:non_exist}, we belive this is closely related to the transition
between the ``small amplitude regime'', and an essentially nonlinear
regime.

\begin{figure}[h]
  \includegraphics[width=3in]{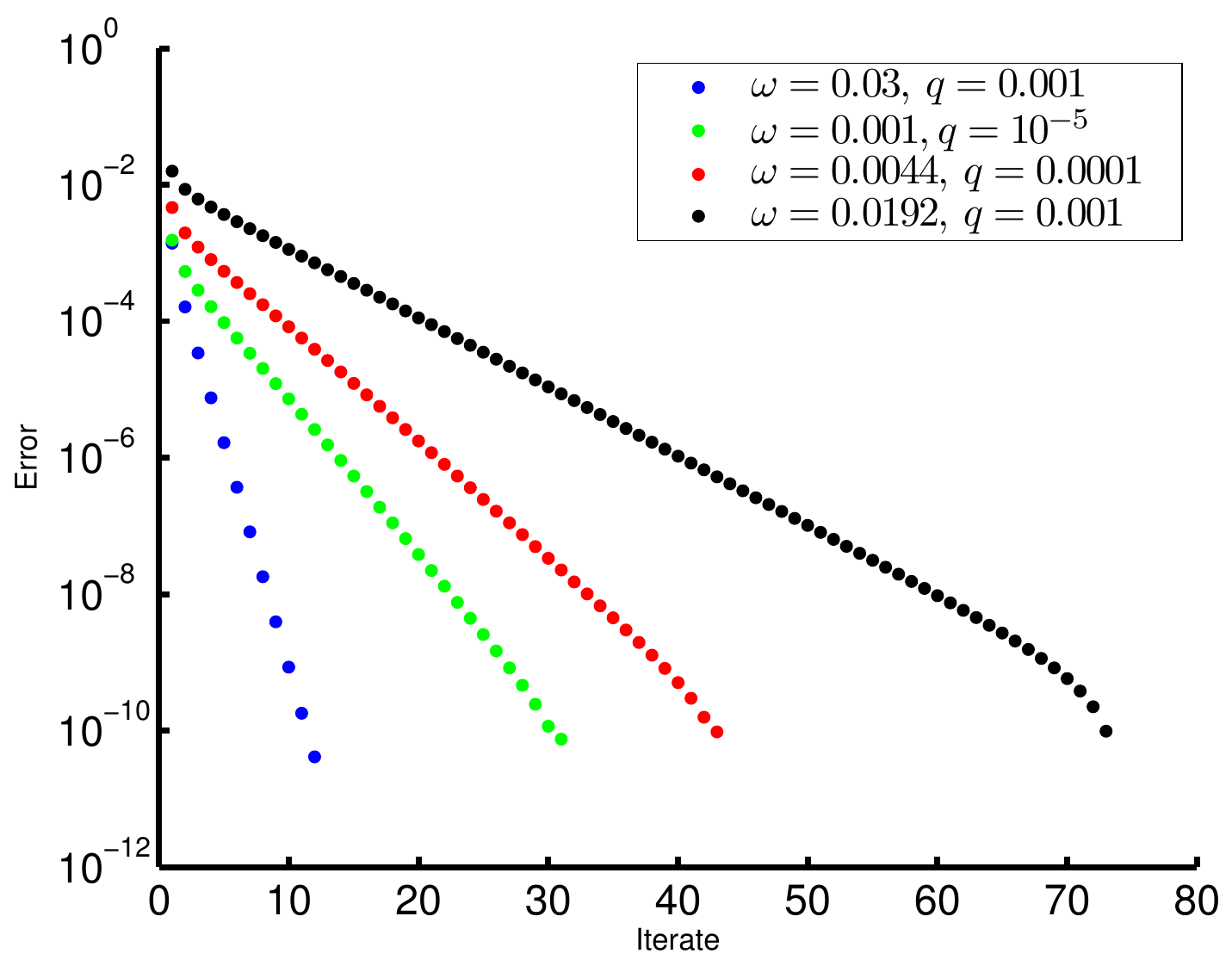}
  \caption{$L^{\infty}$ error between approximate solutions and the
    final iterate, at 512 grid points, computed to a tolerance of
    $10^{-10}$.}
  \label{f:iterate_errors}
\end{figure}

\begin{figure}[h]
  \begin{center}
    \includegraphics[width=2.47in]{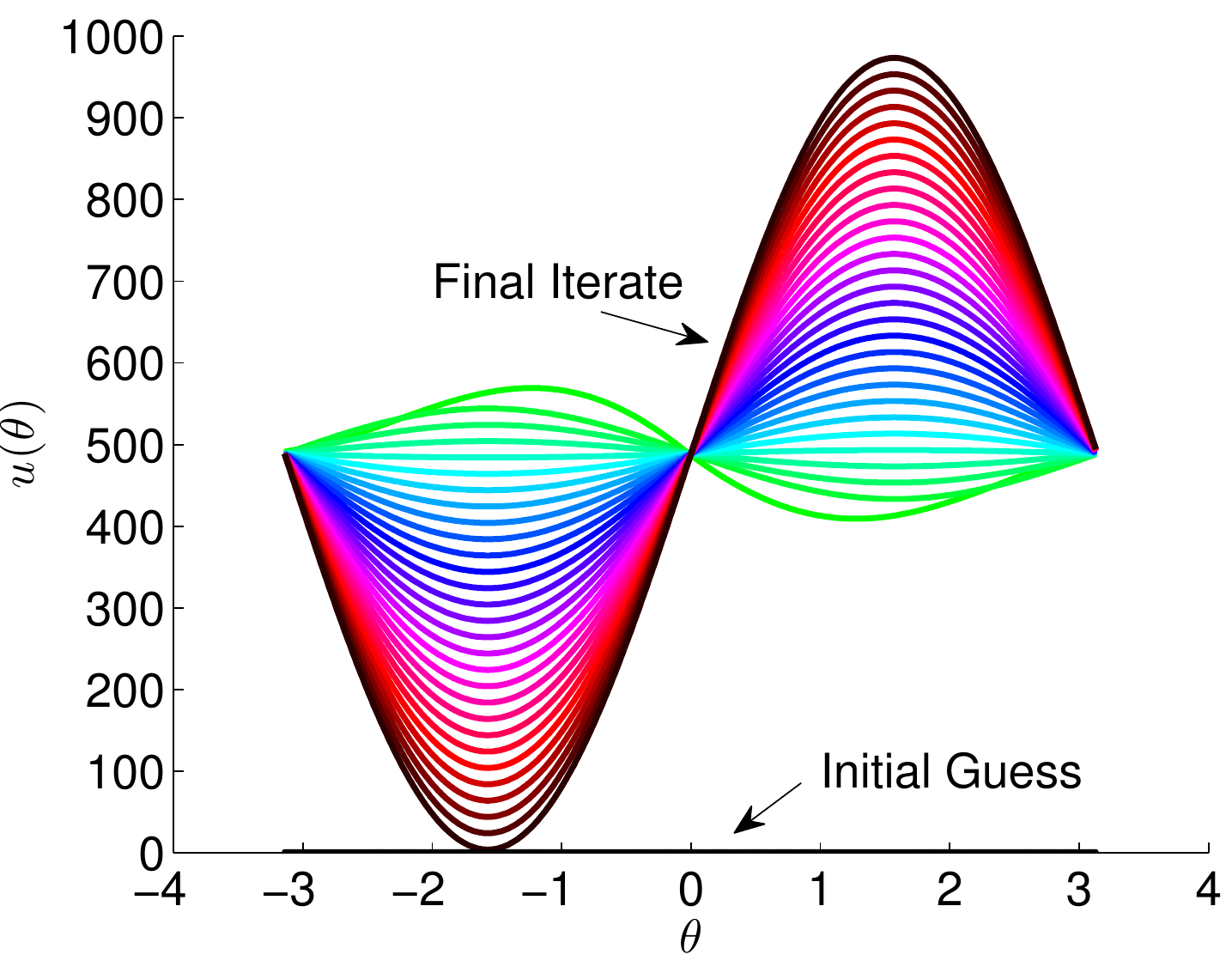}
  \end{center}
  \caption{A typical divergent sequence of iterates, caused by the
    selection of $\omega = 0.1$ and $q=0.1$.  The iterative spectral
    method was run to only 30 iterations -- after 50 iterations, the
    solution had a maximum of $4\times 10^5$.}
  \label{f:diverge}
\end{figure}

However, when we do observe convergence, there is a very high degree
of precision, even using a modest number of grid points. Figure
\ref{f:errors}, shows how quickly the maximum and minimum of the
numerical solution converge to one with $10^5$ grid points as the
number of grid points increases.  Again, solutions with values of $q$
and $\omega$ closer to \eqref{e:nonexist} converge less rapidly; the
solution in black is much closer to this threshold than the other
three.

\begin{figure}[h]
  \centering \subfigure[The error in the maximum of the solution.]
  {\includegraphics[width=2.45in]{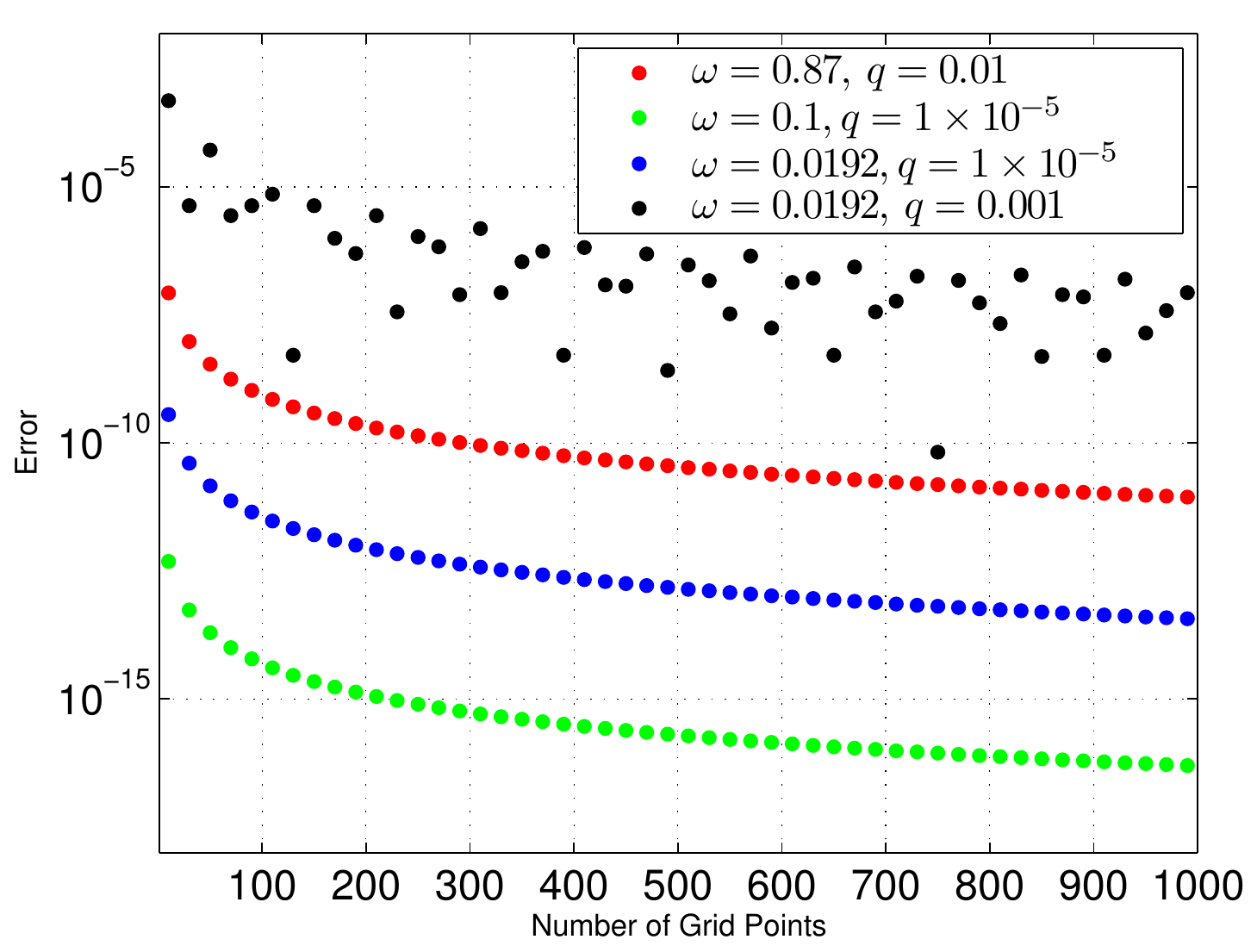}} \subfigure[The
  error in the minimum of the solution.]
  {\includegraphics[width=2.45in]{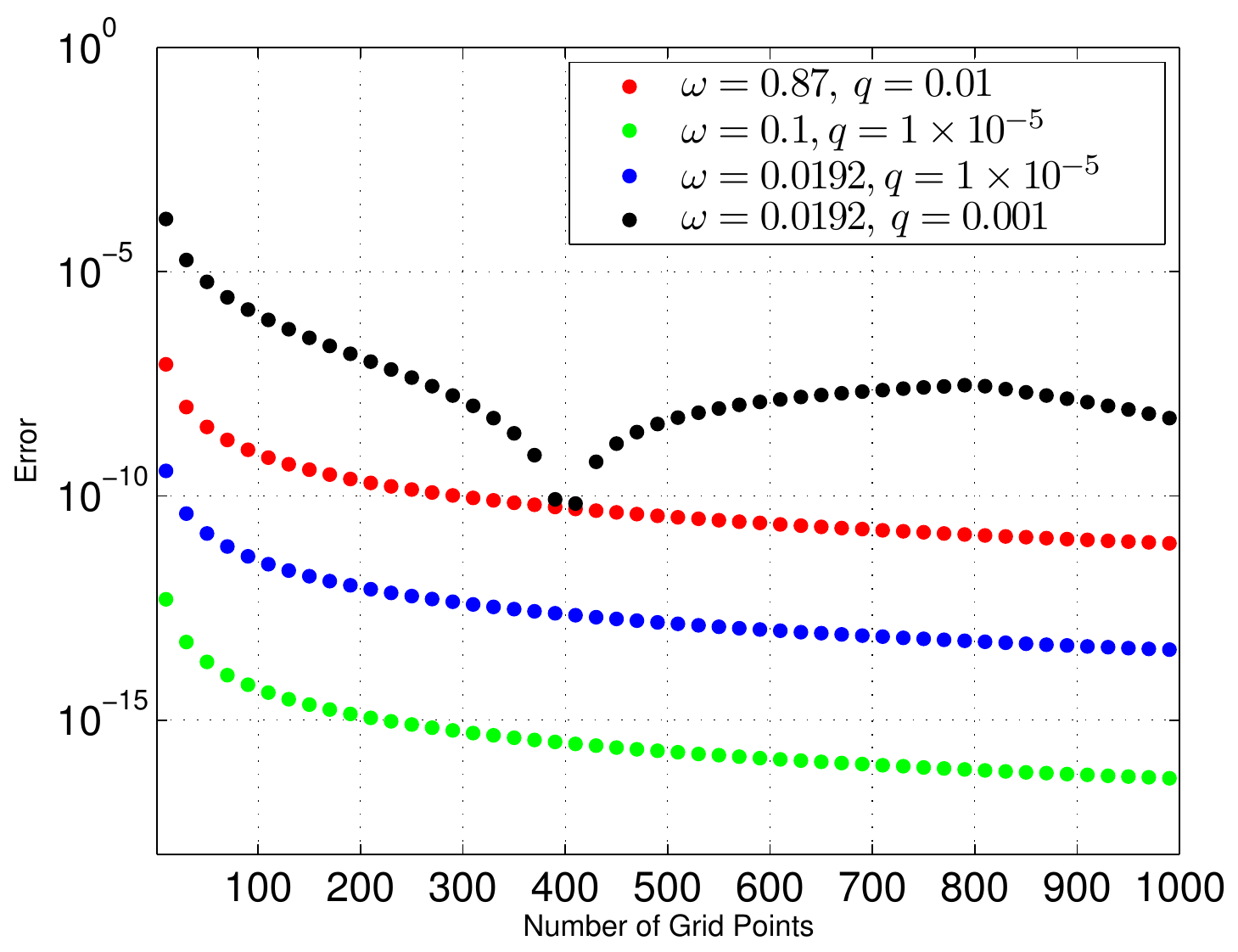}}
  \caption{Convergence of the iterative spectral method, as measured
    by an error in the minimum and maximum, as the grid size is
    increased.}
  \label{f:errors}
\end{figure}

\subsection{Relation to Other Numerical Methods}

Steady state solutions of \eqref{eq:tfe} have been computed in several
ways.  Some groups, such as Ashmore, Hosoi, \& Stone,
\cite{Ashmore:2003p8741}, have run the time dependent equation to
steady state.  More recently, Benilov, Benilov, \& Kopteva computed
the solutions directly by a finite difference discretization with a
Newton algorithm, \cite{Benilov:2008p4148}.

An advantage of our approach is that it converges quickly and is
extremely easy to program. An implementation of the algorithm in
\matlab is available at \url{http://hdl.handle.net/1807/30015}.

Our approach bears some resemblance to Petviashvilli's algorithm and
the related spectral renormalization procedure, which have become
quite popular within the nonlinear dispersive wave community for
computing solitons,
\cite{Ablowitz:2005p9,Pelinovsky:2005p14,Petviashvili:1976p6194}.  For
such equations, including Korteweg - de Vries and nonlinear
Schr\"odinger/Gross-Pitaevskii, solitary wave solutions solve
semilinear elliptic equations of the form
\begin{equation}
  -\nabla^2 Q + \lambda Q = f(Q),\quad \lambda >0, \quad Q:\R^d \to \R.
  \label{e:ndwe}
\end{equation}
For power nonlinearities, $f(Q) = Q^p$, Petviashvilli's algorithm
finds a solution by seeking a fixed point of an iteratively
renormalized nonlinear equation in the Fourier domain,
\begin{equation}
  \widehat{Q}^{(j+1)} = \Lambda_{j} \frac{\widehat{f}(Q^{(j)})}{\lambda +
    \abs{{\boldsymbol{k}}}^2}, \quad \Lambda_{j} = \frac{\int (\lambda +
    \abs{{\boldsymbol{k}}}^2) \abs{\widehat{Q}^{(j)}}^2 d{\bf k}  }{\int
    \widehat{f}(Q ^{(j)}) \overline{\widehat{Q}}^{(j)} d{\bf k}    }
  \label{e:semilinear_petvia}
\end{equation}
$Q^{(j)}$ is the $j$-th approximation, and $\Lambda_j$ is a
renormalization constant at this iterate.

Our scheme, \eqref{eq:pvviter} lacked any such renormalization
constant, $\Lambda_j$.  This raises the question of whether or not
such a constant could improve our approach.  Our results show this is
not the case.  Introducing the constant
\begin{equation}
  \Lambda_j \equiv \frac{ \sum
    \left[-ik^3+ik+\tfrac{\omega^4}{q^3}\right]\abs{\ft{v}_k^{(j)}}^2}{ \sum \ft{F}_k^{(j)}\overline{\ft{v}}_k^{(j)}},
  \label{eq:norm}
\end{equation}
we can alter our iteration scheme to be
\begin{equation}
  \ft{v}_{k}^{(j+1)} =\Lambda_j^\gamma\paren{\frac{q^3}{\omega^4}
    \frac{\delta_{k,1} - 
      \delta_{k,-1}}{2i}+
    \frac{\ft{F}_k^{(j)} }{-ik^3+ik+\tfrac{\omega^4}{q^3}}}.
  \label{eq:iter}
\end{equation}
where $\gamma$ is a constant which we may alter to stabilize or
accelerate the scheme.

An outcome of our experiments is that not only does our method
converge with $\gamma =0$, there does not appear to be appreciable
benefit with any other value; see Table \ref{t:conv_results}.  For
this reason, we do not employ $\Lambda_j$, and only deem our algorithm
an iterative spectral method.

\begin{table}[h!]

  \caption{Number of iterations required for
    $\norm{u_{n+1} - u_n}\leq 10^{-10}$i for a variety of initial
    guesses $u_0$.
    Computed on $2^{12}$ grid points with different values of
    $\gamma$.  $q = 0.001153$ and $\omega = 0.2$ in all
    cases.}
  \label{t:conv_results}

  \begin{tabular}{ c c | c  c  c  c  c  c  c }
    & & \multicolumn{7}{|c}{$\gamma$} \\
    & & -0.75 & -0.50 & -0.25 & 0 & 0.25 & 0.50 & 0.75\\	  
    \hline 
    \multirow{3}{*}{$u_0$} & $ 2\frac{q}{\omega}$ & 42 & 23 & 16 & 12 & 20 & 31 & 69 \\
    & $\frac{q}{\omega} + \frac{q^3}{\omega^4} \cos(\theta)$ & 2 & 2 & 2 & 2 & 2 & 2 & 2\\ 
    & $\frac{q}{\omega} + \frac{q^3}{\omega^4}{\rm sech}(\theta)$ & 38 & 18 & 10 & 3 & 9 & 17 & 39 
  \end{tabular}
\end{table}

\subsection{Mass Constraints and Non-Uniqueness}
\label{s:nonunique}
Given values of $q$ and $\omega$, our algorithm successfully converges
to non-trivial solutions.  However, for a given value of $q$ and
$\omega$, there may be more than one solution of \eqref{eq:steadyq}.
To study this phenomenon, it is helpful to introduce the {\it mass} of
a solution
\begin{equation}
  M \equiv \int_{-\pi}^\pi u(\theta) d\theta.
\end{equation}
Benilov, Benilov \& Kopteva, \cite{Benilov:2008p4148} tried fixing
$\omega$, and numerically explored the relationship between the flux
and the mass as they varied.  They found non-uniqueness of the
solutions, in the sense that for a certain range of $M$, there were
multiple solutions with different values of $q$.  Alternatively, for a
certain range of $q$, there were multiple solutions with different
values of $M$.  See Figure 14 of their work for a phase diagram, and
see our Figure \ref{f:triple_diagram} for a similar diagram.

This non-uniqueness affects our algorithm.  For a given $(\omega, q)$,
the solution that the method converges to may not be the solution we
desire, as the following example shows.  Using, AUTO, \cite{auto07},
if we seek a solution with $M=1$ and $\omega = .09$, we will obtain
the solution pictured in Figure \ref{f:multi_pt} (a).  This solution
has flux parameter $q = 0.0114039$.

Suppose we now fix $\omega = .09$, and continue the solution by
lowering the mass.  As seen Figure \ref{f:multi_pt} (b), there will be
a second solution, with the same $q$, but slightly smaller mass.  For
$q =0.01140392 $, the second solution is at $M = 0.912782$ Thus, for a
given $(\omega, q)$, there may be multiple solutions, each with a
different mass-- which solution does the iterative spectral algorithm
converge to?  As Figure \ref{f:multi_pt} (a) shows, it converges to
the solution with the smaller mass.  The agreement between the AUTO
solution and our method's solution is quite good; after the AUTO
solution is splined onto the regular grid of the spectral solution,
the pointwise error is $\bigo(10^{-7})$.


\begin{figure}
  \subfigure[Two different solutions with $\omega = .09$ and
  $q=0.0114039$]{\includegraphics[width=2.45in]{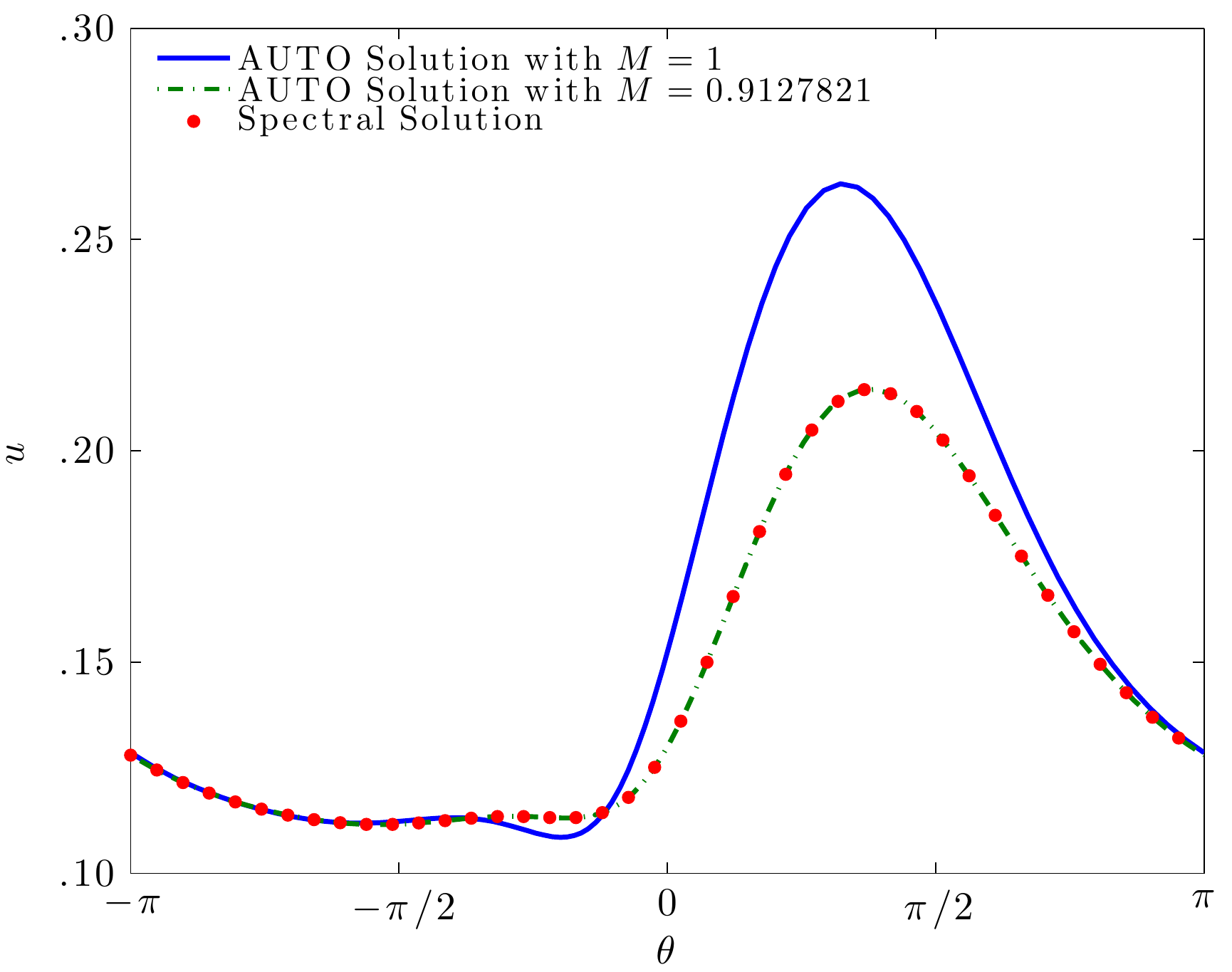}}
  \subfigure[$(M, q)$ pairs of solutions with $\omega =
  .09$]{\includegraphics[width=2.45in]{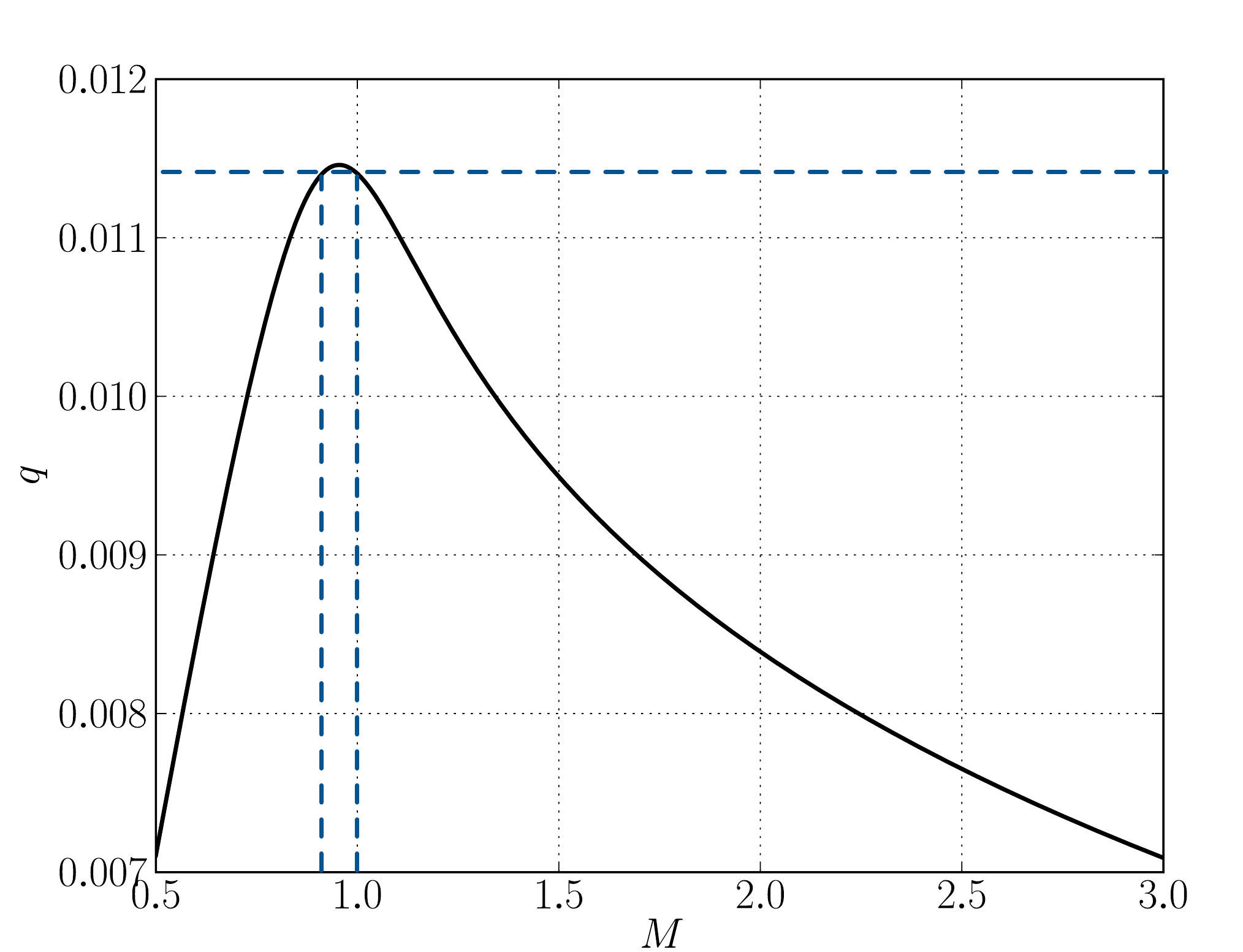}}
  \caption{AUTO finds two solutions by continuation methods the same
    $\omega$ and $q$, but with different mass.  The Spectral Iterative
    Method converges to the solution with the smaller mass.}
  \label{f:multi_pt}
\end{figure}

The inability of our algorithm to recover the $M=1$ solution can be
traced to the development of a linear instability in the iteration
scheme.  Let $v^{(j)} = v + p^{(j)}$, where $v$ is a solution, and
$p^{(j)}$ is a perturbation at iterate $j$.  If we linearize
\eqref{e:iter} about $v$, we have
\begin{equation}
  \label{e:linearized_iter}
  \begin{split}
    p^{(j+1)} &= L^{-1} \bracket{ \frac{\omega^5}{q^3}\frac{6 q^3 + 6
        q^2 \omega v + 4 q \omega^2 v^2 + \omega^3 v^3}{(q + \omega
        v)^4}
      v p^{(j)}}\\
    & =L^{-1} \paren{F'(v) p^{(j)}} \equiv \tilde{L}_{v} p^{(j)}.
  \end{split}
\end{equation}
If the spectrum of $\tilde{L}_v$ lies strictly inside the unit circle,
the perturbation will vanish, and the solution $v$ will be stable with
respect to the iteration.  However, if there are points in the
spectrum of this operator which lie outside the unit circle, the
solution will be unstable to the iteration, and we would expect our
algorithm to diverge.

Though we shall not attempt to prove properties of the spectrum of
$\tilde{L}$, we can discretize our problem and compute eigenvalues of
the induced matrix.  Computing the eigenvalues of a discrete
approximation of $\tilde{L}_v$ at the two AUTO solutions, we find that
the matrix corresponding to the $M=1$ solution does, in fact, have an
eigenvalue outside the unit circle, $\lambda_{\rm unstable} =
1.017059$; see Figure \ref{f:linspec_diagram}.  The discretized
operator is defined as
\begin{equation} {\bf \tilde{L}}_{\bf v} \paren{D^{3} + D +
    \tfrac{\omega^4}{q^3} I }^{-1} \diag\set{F'({\bf v})}
\end{equation}
with the differentiation matrix $D$, a dense Toeplitz matrix for band
limited interpolants, as in \cite{trefethen2000spectral}.  The value
of the unstable eigenvalue converges quite quickly as we increase the
number of grid points, as shown in Table
\ref{e:eigenvalue_convergence}.

\begin{figure}
  \includegraphics[width=3in]{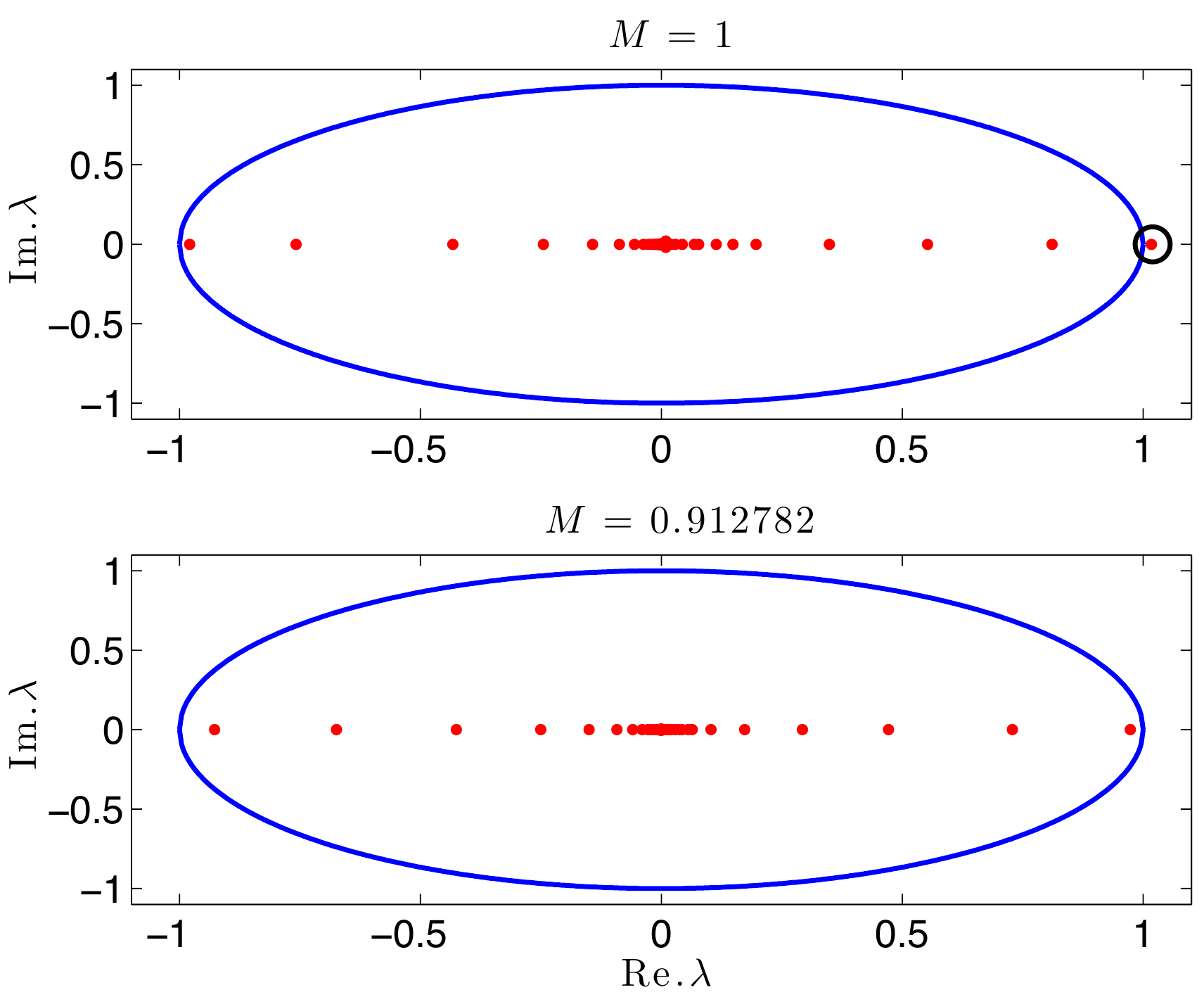}
  \caption{The eigenvalues of the discretized $\tilde{L}_v$ for the
    two solutions computed using AUTO with $\omega =.09$ and $q =
    .0114039$.  Note that for the $M=1$ solution, there is an
    eigenvalue outside the unit circle at $\lambda =
    1.017059$. Computed with 2048 grid points.}
  \label{f:linspec_diagram}
\end{figure}

\begin{table}
  \caption{Values of the unstable eigenvalue of ${\bf \tilde{L}}_{\bf
      v}$ for the $\omega =.09$, $q =
    .0114039$ and $M=1$ solution as a function of the number of grid
    points.}
  \label{e:eigenvalue_convergence}
  \begin{tabular}{r | l}
    No. of Grid Points & $\lambda_{\rm unstable}$\\
    \hline
    16 &1.01711308844\\
    64 & 1.01705921834\\
    256 &1.01705919446 \\
    2048 &1.01705919124 
  \end{tabular}
\end{table}

The $(M,q)$ phase space of solutions with fixed $\omega$ can be more
tortuous than that shown in Figure \ref{f:multi_pt} (b).  Indeed, as
we send $\omega \to 0$, a singular limit, there may not just be
multiple solutions of {\it different mass and the same flux}, but also
multiple solutions of {\it different flux and the same mass}.  Such an
example appears in Figure \ref{f:triple_diagram}.  These cases are
more extensively explored in \cite{Benilov:2008p4148} and in Badali
{\it et al}, \cite{Badali:2011p14435}.

\begin{figure}
  \includegraphics[width=3in]{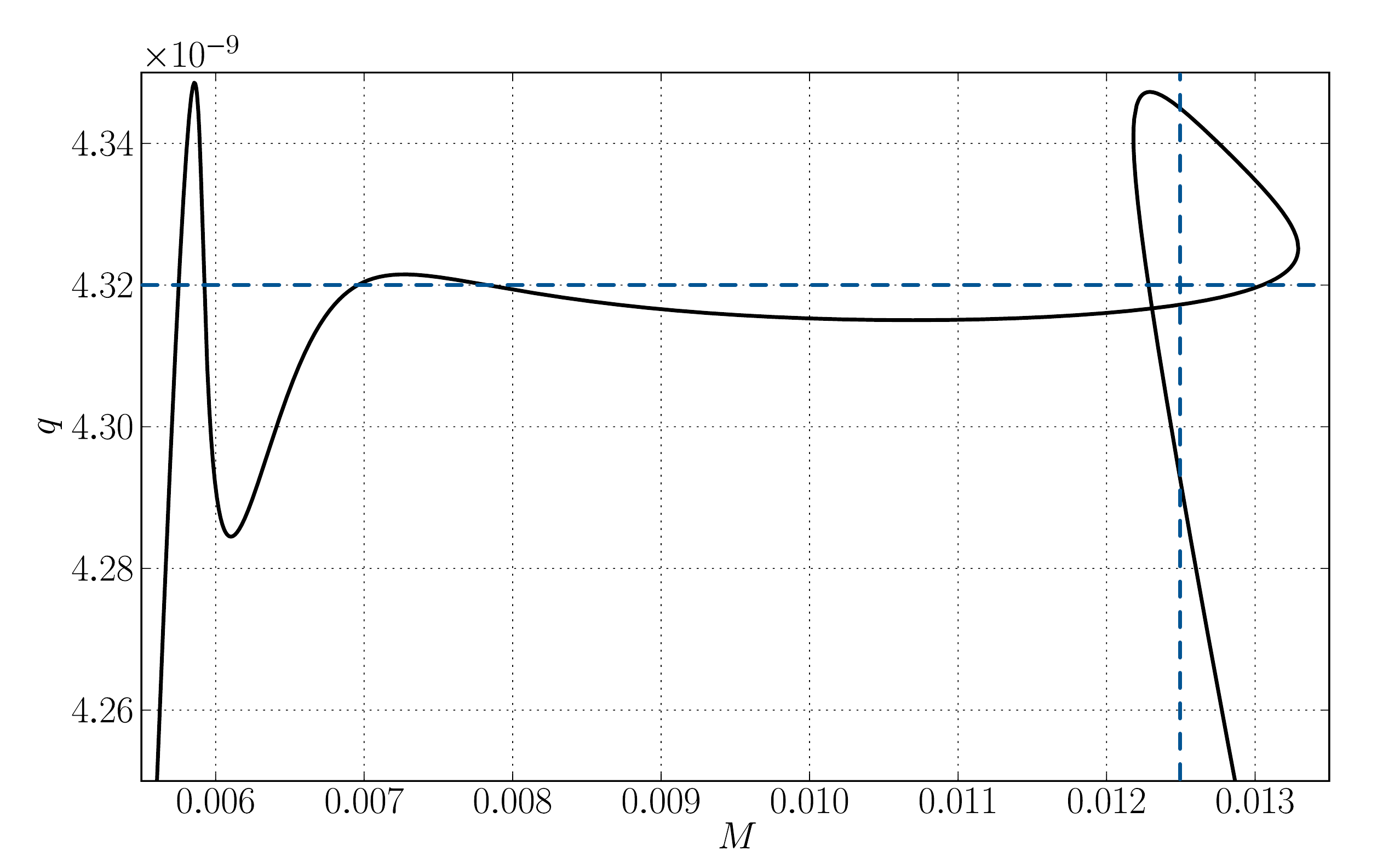}
  \caption{The $(M,q)$ phase space of solutions with rotation
    parameter $\omega = 5\times 10^{-6}$.  Note that there can be as
    many as six solutions of different mass and the same flux, and as
    many as three solutions of different flux and the same mass.
    Computed using AUTO.}
  \label{f:triple_diagram}
\end{figure}

\section{Non-Existence Results}
\label{s:non_exist}

As mentioned in the introduction, in \cite{Chugunova:2010p6065}, the
authors prove that solutions to \eqref{eq:steadyq} do not exist for
all $(\omega, q)$ pairs.  Indeed, they show there are no solutions
when \eqref{e:nonexist} holds.  Interestingly, this relation is stated
entirely in terms of $q$ and $\omega$, and makes no consideration of
the mass, $M$.

We now explore, numerically, \eqref{e:nonexist}. For fixed mass, we
compute the associated solution at a variety of $\omega$ values,
determining $q$.  The results are appear in Figure \ref{f:nonexist}.
As $(\omega, q) $ cannot uniquely determine $u$, we rely on AUTO and
its continuation algorithms, rather than our own iterative scheme.

\begin{figure}
  \centering
  \includegraphics[width=5in]{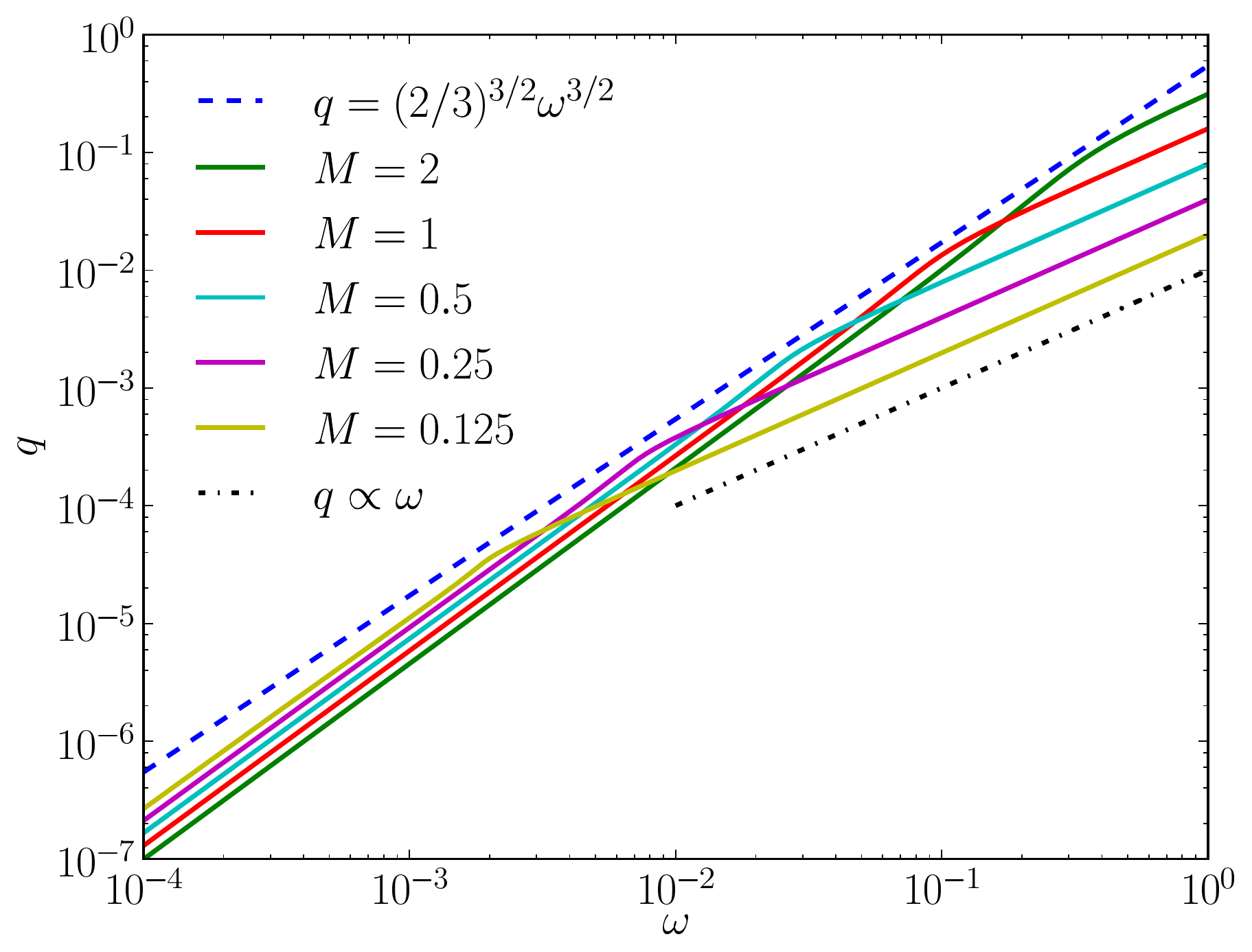}
  \caption{The $(\omega, q)$ phase space of solutions to
    \eqref{eq:steadyq} at prescribed mass values.  As the figure
    shows, even two parameters are insufficient to uniquely determine
    a steady state solution; where the curves of constant mass cross,
    there will be two or more solutions with the same $\omega$ and $q$
    values.  All curves of solutions are constrained by
    \eqref{e:nonexist}, as expected.  Also note that at fixed $q$, for
    sufficiently large $\omega$, the relationship is nearly linear.
    Computed using AUTO.}
  \label{f:nonexist}
\end{figure}

Examining the figure, we remark on the two features found in each
fixed mass curve. Far to the right of the non-existence line, the
$\omega$-$q$ relationship appears to be linear.  Indeed, for moderate
mass values, the relationship appears to be of the form
\eqref{e:puk_regime}.  It is in this ``small amplitude'' regime that
our iterative spectral algorithm succeeds.

As $\omega$ continues to decrease, the scaling changes to accomodate
\eqref{e:nonexist}, and appears to follow a curve like $ q \propto
\omega^{\beta}$, for a value of $\beta>3/2$.  For all masses, the
deviation from \eqref{e:nonexist} is less than an order of magnitude.
Thus, there is not an obvious small parameter from which to formulate
a series expansion.  We also remark that there are many intersections
between the curves of constant mass; the lack of uniqueness of a
solution for a given $(\omega, q)$ is quite common.  As $\omega$ and
$q$ approach this essentially nonlinear regime, our algorithm becomes
limited.  As discussed in the preceding, section unstable eigenvalues
can appear, precluding convergence to solutions with a desired mass,
and the rate of convergence slows.

Turning to Figure \ref{f:thick}, we see the tendency of the solutions
to form singularities, vertical asymptotes, as $\omega$ tends towards
zero for fixed $q$.  These were computed using our spectral iterative
algorithm; the asymptotes occurred where the algorithm failed to
converge in a reasonable number of iterations. As these figures show,
it is possible to get quite close to the theoretical threshold of
\eqref{e:nonexist}.  This suggests that the constraint may be sharp,
and this warrants further investigation.  Even if the constant is not
correct, the scaling appears to be sharp, and the inequality plays an
important role in separating two physical regimes.

\begin{figure}[h]
  \centering \subfigure[$q = 1\times 10^{-3}$. The mass of the final
  solution was
  0.3951]{\includegraphics[width=2.4in]{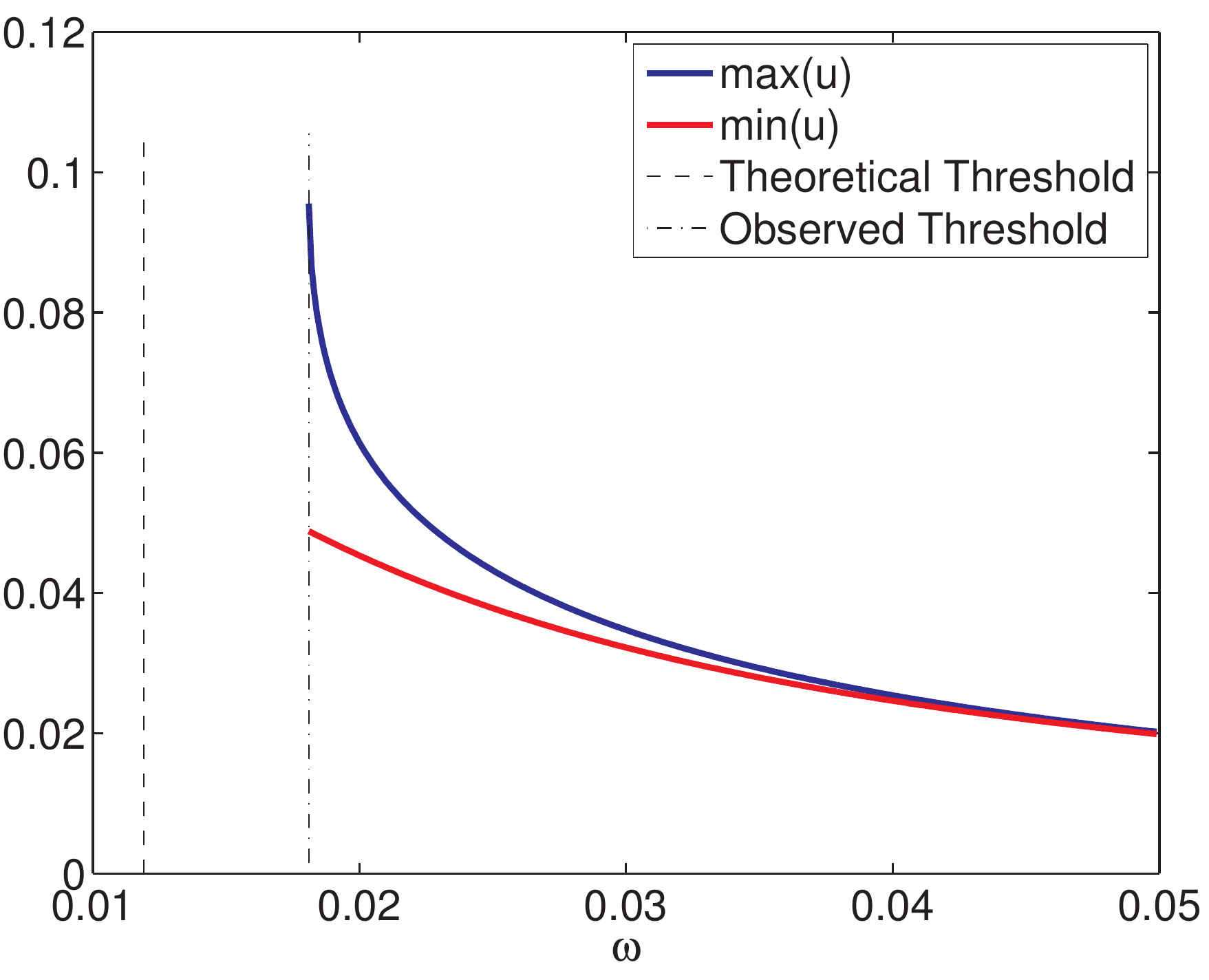}}
  \subfigure[$q = 1\times 10^{-4}$.  The mass of the final solution
  was 0.1612]{\includegraphics[width=2.4in]{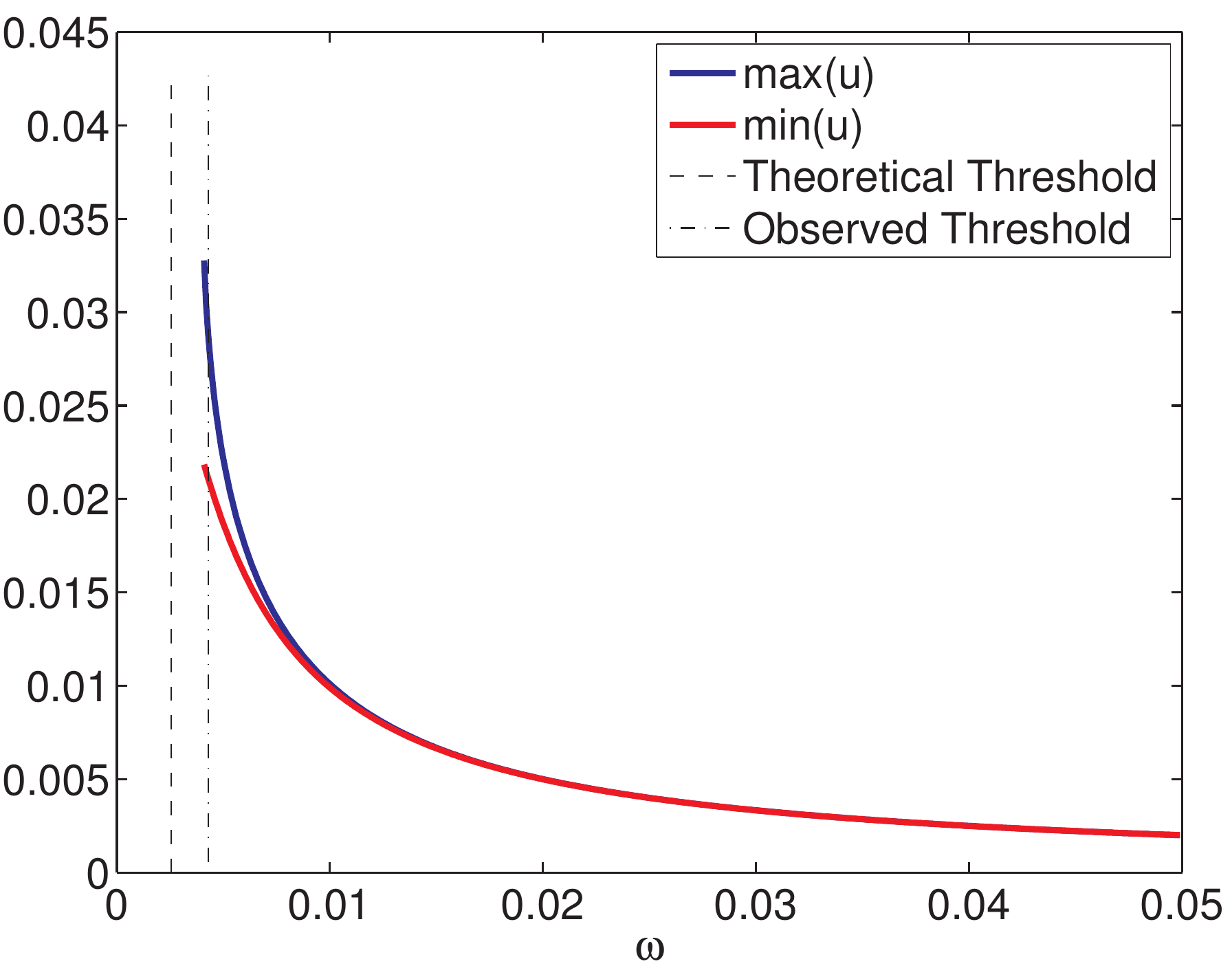}}
	
  \caption{Thickness of solutions with a grid size of $2^{10}$ as
    $\omega$ is varied.  The dashed lines denote the observed
    threshold at which our spectral iterative algorithm fails and the
    theoretical threshold, \eqref{e:nonexist} beyond which no
    solutions should exist. Computed using the spectral iterative
    method.}
  \label{f:thick}
\end{figure}

\section{Discussion}
\label{s:discussion}

In this work, we have given a simple proof of the positivity of steady
state solutions of a thin film equation with rotation, subject to
modest assumptions, provided $n\geq 2$.  It may be possible to employ
additional regularity properties of the steady state solution to lower
this threshold; we only relied on Sobolev embeddings in our result.

We have also formulated a very easy to implement algorithm for
computing the steady state solutions for a given $(\omega, q$) pair.
However, in cases where multiple solutions are known to exist, an
instability can develop about one of the solutions, precluding
convergence to this solution.  We conjecture that the preference is
related to the proximity to the non-existence curve, where the
parameters have left the small amplitude regime.  Finding a way to
stabilize the spectral algorithm in this regime remains an open
problem.

Finally, our computations suggest further exploration of the
non-existence curve is warranted.  Figures \ref{f:nonexist} and
\ref{f:thick} suggest that the nonexistence relationship may be sharp.
Furthermore, the phase diagram indicates that this inequality
separates two physical regimes.  In the first, where $\omega$ and $q$
are nearly linearly related, the solution is nearly constant, and can
be approximated by \eqref{e:approx_steady}.  In order to accomodate
the nonexistence constraint, the phase diagram bends, and we enter a
regime where the parameters are nonlinearly related, and the steady
state solutions cease to be well described by unimodal bumps.

{\bf Acknowledgements:} {The authors wish to thank A. Buchard and
  M. Chugunova for suggesting this problem and for many helpful
  discussions throughout this project's development.  This work was
  supported in part by NSERC Discovery Grant \# 311685-10 and an NSERC
  Undergraduate Student Research Award.}

\bibliographystyle{abbrv} \bibliography{film}

\end{document}

%% file: dgmacros.tex
\newcommand{\ut}{u_{\theta\theta\theta}}

\newcommand{\fns}{C_{\mathrm{per}}^{\infty}(\Omega)}

\newcommand{\mint}[1]{\int_{-\pi}^{\pi}\!#1 \,d\theta}
\newcommand{\term}{u_{\theta\theta}+u - \sin(\theta)}

\newcommand{\ft}[1]{\widehat{#1}}

%% file: gsmacros.tex
\newcommand{\abs}[1]{\lvert#1\rvert}
\newcommand{\norm}[1]{\lVert#1\rVert}
\newcommand{\paren}[1]{\left(#1\right)}
\newcommand{\bracket}[1]{\left[#1\right]}
\newcommand{\set}[1]{\left\{#1\right\}}

\newcommand{\per}{\mathrm{per}}
\newcommand{\bigo}{\mathrm{O}}
\newcommand{\matlab}{\textsc{Matlab} }

\newcommand{\eps}{\epsilon}
\newcommand{\vareps}{\varepsilon}

\newcommand{\diag}{\mathrm{diag}}

\newcommand{\R}{\mathbb{R}}

\newcommand{\supp}{\mathrm{supp}\:}

\newtheorem{prop}{Proposition}

\newtheorem{lem}{Lemma}

%% file: thin_film_paper.bbl
\begin{thebibliography}{10}

\bibitem{Ablowitz:2005p9}
M.~J. Ablowitz and Z.~H. Musslimani.
\newblock Spectral renormalization method for computing self-localized
  solutions to nonlinear systems.
\newblock {\em Optics letters}, 30(16):2140--2142, 2005.

\bibitem{adams1975sobolev}
R.~Adams and J.~Fournier.
\newblock {\em Sobolev Spaces}.
\newblock Academic Press, 2003.

\bibitem{Ashmore:2003p8741}
J.~Ashmore, A.~Hosoi, and H.~Stone.
\newblock The effect of surface tension on rimming flows in a partially filled
  rotating cylinder.
\newblock {\em Journal of Fluid mechanics}, 479:65--98, Jan 2003.

\bibitem{Badali:2011p14435}
D.~Badali, M.~Chugunova, and D.~Pelinovsky{\ldots}.
\newblock Regularized shock solutions in coating flows with small surface
  tension.
\newblock {\em Physics of Fluids}, 23(9):093103, Jan 2011.

\bibitem{Becker:2005p8459}
J.~Becker and G.~Gr\"{u}n.
\newblock {The thin-film equation: recent advances and some new perspectives}.
\newblock {\em Journal of Physics: Condensed Matter}, 17(9):S291--S307, Mar.
  2005.

\bibitem{Benilov:2008p4148}
E.~Benilov, M.~Benilov, and N.~Kopteva.
\newblock Steady rimming flows with surface tension.
\newblock {\em Journal of Fluid mechanics}, 597:91--118, Jan 2008.

\bibitem{Beretta:1995p13153}
E.~Beretta, M.~Bertsch, and R.~D. Passo.
\newblock Nonnegative solutions of a fourth-order nonlinear degenerate
  parabolic equation.
\newblock {\em Arch. Rational Mech. Anal.}, 129(2):175--200, 1995.

\bibitem{Bernis1990}
F.~Bernis and a.~Friedman.
\newblock {Higher order nonlinear degenerate parabolic equations}.
\newblock {\em Journal of Differential Equations}, 83(1):179--206, Jan. 1990.

\bibitem{Bertozzi:1994p10686}
A.~L. Bertozzi and M.~C. Pugh.
\newblock The lubrication approximation for thin viscous films: the moving
  contact line with a ``porous media'' cut-off of van der waals interactions.
\newblock {\em Nonlinearity}, 7(6):1535--1564, 1994.

\bibitem{Burchard:2010p8420}
A.~Burchard, M.~Chugunova, and B.~K. Stephens.
\newblock Lyapunov's method proves convergence to equilibrium for a thin film
  equation.
\newblock {\em arXiv}, math.AP, Nov 2010.

\bibitem{Burchard:2010p8616}
A.~Burchard, M.~Chugunova, and B.~K. Stephens.
\newblock On the energy-minimizing steady states of a thin film equation.
\newblock {\em arXiv}, math.AP, Sep 2010.
\newblock 14 pages, 5 figures.

\bibitem{Chugunova:2010p6065}
M.~Chugunova, M.~Pugh, and R.~Taranets.
\newblock Nonnegative solutions for a long-wave unstable thin film equation
  with convection.
\newblock {\em SIAM Journal on Mathematical Analysis}, 42(4):1826--1853, Jan
  2010.

\bibitem{Craster:2009p8507}
R.~Craster and O.~Matar.
\newblock Dynamics and stability of thin liquid films.
\newblock {\em Reviews of Modern Physics}, 81(3):1131--1198, Jan 2009.

\bibitem{auto07}
E.~Doedel, B.~Oldeman, and et~al.
\newblock {AUTO-07P}.

\bibitem{Evans:1998fk}
L.~Evans.
\newblock {\em Partial Differential Equations}, volume~19 of {\em Graduate
  Studies in Mathematics}.
\newblock American Mathematical Society, 1998.

\bibitem{Oron:1997p8421}
A.~Oron and S.~G. Bankoff.
\newblock {Long-scale evolution of thin liquid films}.
\newblock {\em Reviews of Modern Physics}, 69(3):931--980, July 1997.

\bibitem{Pelinovsky:2005p14}
D.~E. Pelinovsky and Y.~Stepanyants.
\newblock Convergence of petviashvili's iteration method for numerical
  approximation of stationary solutions of nonlinear wave equations.
\newblock {\em SIAM Journal on Numerical Analysis}, 42(3):1110--1127, 2005.

\bibitem{Petviashvili:1976p6194}
V.~Petviashvili.
\newblock Equation of an extraordinary soliton.
\newblock {\em Soviet Journal of Plasma Physics}, Jan 1976.

\bibitem{Pougatch:2011p14442}
K.~Pougatch and I.~Frigaard.
\newblock Thin film flow on the inside surface of a horizontally rotating
  cylinder: Steady state solutions and their stability.
\newblock {\em Physics of Fluids}, 23(2):022102, Jan 2011.

\bibitem{Pukhnachev1977}
V.~Pukhnachev.
\newblock {Motion of a liquid film on the surface of a rotating cylinder in a
  gravitational field}.
\newblock {\em Journal of Applied Mechanics and Technical Physics},
  18(3):344--351, 1977.

\bibitem{pukhnachov2004asymptotic}
V.~Pukhnachov.
\newblock Asymptotic solution of the rotating film problem.
\newblock {\em Izv. Vyssh. Uchebn. Zaved. Severo-Kavkaz. Reg. Estestv. Nauk,
  ``Mathematics and Continuum Mechanics''(a special issue)}, pages 191--199,
  2004.

\bibitem{Pukhnachov:2005p5033}
V.~Pukhnachov.
\newblock On the equation of a rotating film.
\newblock {\em Siberian Mathematical Journal}, 46(5):913--924, 2005.

\bibitem{Shishkov:2004p13213}
A.~E. Shishkov and R.~M. Taranets.
\newblock On the equation of the flow of thin films with nonlinear convection
  in multidimensional domains.
\newblock {\em Ukr. Mat. Visn.}, 1(3):402--444, 447, 2004.

\bibitem{Taranets:2006p13207}
R.~M. Taranets and A.~E. Shishkov.
\newblock A singular cauchy problem for the equation of the flow of thin
  viscous films with nonlinear convection.
\newblock {\em Ukra{\"\i}n. Mat. Zh.}, 58(2):250--271, 2006.

\bibitem{trefethen2000spectral}
L.~Trefethen.
\newblock {\em Spectral methods in MATLAB}, volume~10.
\newblock SIAM, 2000.

\end{thebibliography}
